\documentclass[12pt,a4paper]{article}

\usepackage{latexsym,amsfonts,amsmath,amsthm,graphicx,amssymb,xcolor}

\usepackage[colorlinks=true,linkcolor=blue,citecolor=blue]{hyperref}

\DeclareMathOperator{\supp}{supp}

\newtheorem{theorem}{Theorem}
\newtheorem{lemma}{Lemma}
\newtheorem{corollary}{Corollary}

\newtheorem{remark}{Remark}

\newcommand{\satop}[2]{\stackrel{\scriptstyle{#1}}{\scriptstyle{#2}}}

\newcommand{\bsgamma}{\boldsymbol{\gamma}}

\newcommand{\bsk}{\boldsymbol{k}}

\newcommand{\bsx}{\boldsymbol{x}}
\newcommand{\bsh}{\boldsymbol{h}}
\newcommand{\bsg}{\boldsymbol{g}}

\newcommand{\bsm}{\boldsymbol{m}}
\newcommand{\bsw}{\boldsymbol{w}}

\newcommand{\uu}{\mathfrak{u}}
\newcommand{\vv}{\mathfrak{v}}

\newcommand{\bsy}{\boldsymbol{y}}

\newcommand{\sob}{{\rm sob}}
\newcommand{\kor}{{\rm kor}}

\newcommand{\bszero}{\boldsymbol{0}}
\newcommand{\rd}{\,\mathrm{d}}
\newcommand{\NN}{\mathbb{N}}
\newcommand{\ZZ}{\mathbb{Z}}

\newcommand{\RR}{\mathbb{R}}

\renewcommand{\pmod}[1]{\,(\bmod\,#1)}

\newcommand{\E}{\mathrm{e}}
\newcommand{\I}{\mathrm{i}}
\newcommand{\sym}{\operatorname{sym}}

\newcommand{\pmu}{\operatorname{(\pm)_{\uu}}} 

\newcommand{\Cpp}{C++}

\allowdisplaybreaks

\begin{document}

\title{Lattice rules for nonperiodic smooth integrands
}

\author{
Josef Dick, Dirk Nuyens,
Friedrich Pillichshammer
}

\date{\today}
\maketitle

\begin{abstract}
The aim of this paper is to show that one can achieve convergence
rates of $N^{-\alpha+ \delta}$ for $\alpha > 1/2$ (and for $\delta > 0$
arbitrarily small) for nonperiodic $\alpha$-smooth cosine series using lattice
rules without random shifting. The smoothness of the functions can be measured by the
decay rate of the cosine coefficients. For a specific choice of the parameters the cosine series space coincides with the unanchored Sobolev space of smoothness $1$.

We study the embeddings of various reproducing kernel Hilbert spaces and numerical integration in the cosine series function space and show that by applying the so-called tent transformation to a lattice rule one can achieve the (almost) optimal rate of
convergence of the integration error. The same holds true for symmetrized lattice rules for the tensor product of the direct sum of the Korobov space and cosine series space, but with a stronger dependence on the dimension in this case.
\end{abstract}

\section{Introduction}

Quasi-Monte Carlo (QMC) rules are equal weight quadrature rules
$$
  \frac{1}{N} \sum_{n=0}^{N-1} f(\bsx_n)
$$
which can be used to approximate integrals of the form
$$
  \int_{[0,1]^s} f(\bsx) \rd \bsx,
$$
see \cite{DP,niesiam,SJ} for more information. In QMC rules, the quadrature points $\{\bsx_0,\bsx_1,\ldots, \bsx_{N-1}\}$ are chosen according to some deterministic algorithm. One can show that the convergence rate of the integration error $\frac{1}{N} \sum_{n=0}^{N-1} f(\bsx_n) - \int_{[0,1]^s}
f(\bsx) \rd \bsx$ depends on the smoothness of the integrand
and some property of the quadrature points.

There are several deterministic construction methods for the quadrature points. One such method yields so-called digital nets. These yield a convergence of the integration error of $N^{-1} (\log N)^s$ for functions of bounded variation \cite{DP,niesiam}.
Higher order digital nets, using an interlacing factor of $d$, yield a
convergence rate of order $N^{-\min(\alpha, d)} (\log N)^{s \alpha}$
for integrands with square integrable partial mixed derivatives of
order $\alpha$ in each variable \cite{D07,D08}.

An alternative to digital nets are lattice rules \cite{niesiam,SJ}. In this case, for a given positive integer $N$ one chooses a vector $\bsg \in \{1,\ldots, N-1\}^s$
called the generating vector and defines the quadrature points by
\begin{equation*}
\left\{\frac{n \bsg }{N}\right\} \quad \text{for } 0 \le n < N.
\end{equation*}
Here, for a real number $x$, the braces $\{ x\}$ denote the fractional part,
i.e., modulo~$1$. For vectors, the fractional part is taken component wise.
It is well known that there are generating vectors for lattice rules
for which the integration error converges with order $N^{-\alpha +
\delta}$ ($\delta > 0$) for smoothness $\alpha \ge 1$, but with the
restriction that $f$ and its partial derivatives up to order $\alpha-1$ in each variable have to be periodic \cite[Theorem~18, p.\ 120]{kor}
and \cite{SKJ,SR}. Fast computer search algorithms for
such vectors are known from \cite{NC06,NC06b}. Hence, in order to be
able to benefit from the fast rate of convergence, one needs to
apply a transformation which makes the integrand (and its partial derivatives) periodic. This can cause some problems though and  is not always recommended
\cite{kor,KSW07}. Since arbitrarily high rates of convergence can be
obtained using digital nets for nonperiodic functions,
the question arises whether this is
also possible for lattice rules. Until now, lattice rules achieve a
convergence rate of at most $N^{-1 + \delta}$ ($\delta >
0$) for nonperiodic integrands (via estimates of the
star-discrepancy \cite{niesiam}). If one applies the so-called
tent transformation and a random shift, this rate of
convergence can be improved to $N^{-2 +\delta}$ for any
$\delta > 0$ for the worst-case error in an unanchored Sobolev space of smoothness $2$, see \cite{H02}.

In this paper we present quadrature rules which achieve a convergence rate of order $N^{-\alpha +
\delta}$, $\delta > 0$, for nonperiodic functions with smoothness
$\alpha > 1/2$. The way we measure smoothness in this paper is slightly
different from the setting used in, for instance, higher order digital
nets \cite{D07,D08} though. We consider functions $f$ which can be represented by a
cosine series.
Note that every continuous function $f \in
L_2([0,1])$ can be represented by a cosine series (see \cite[Theorem~1]{IN2008} for this basis over $[-1,1]$).
This follows from the fact that the functions $$1, \sqrt{2} \cos( \pi x), \sqrt{2} \cos( \pi 2 x), \sqrt{2} \cos( \pi 3 x), \ldots$$ are $L_2$-orthogonal and complete.

As mentioned above, the ``smoothness'' of the cosine series in our context is measured by
the rate of decay of the cosine coefficients. To illustrate this,
consider a one-dimensional function $f:[0,1] \to \RR$ given by
its cosine series
\begin{eqnarray*}
  f(x)
  &=\hphantom{:}&
  \widetilde{f}_{\cos}(0)
  + \sum_{k=1}^\infty \widetilde{f}_{\cos}(k) \sqrt{2} \cos(\pi k x)
  \\
  &=\hphantom{:}&
  \widetilde{f}(0)
  + \sum_{k=1}^\infty \widetilde{f}_{\cos}(2k) \sqrt{2} \cos(\pi 2k x)
  + \sum_{k=1}^\infty \widetilde{f}_{\cos}(2k-1) \sqrt{2} \cos(\pi (2k-1) x)
  \\
  &=:&
  c + f_{\mathrm{per}}(x) + f_{\mathrm{nonper}}(x),
\end{eqnarray*}
where
\begin{align}\label{eq:that}
  \widetilde{f}_{\cos}(k) :=
  \begin{cases}
    \displaystyle \int_0^1 f(x) \rd x & \text{ for } k=0,\\[4mm]
    \displaystyle \int_0^1 f(x) \sqrt{2} \cos(\pi k x) \rd x & \text{ for } k\in \NN.
  \end{cases}
\end{align}
The sum over the even frequencies is a $1$-periodic function
$f_{\mathrm{per}}$ over $[0,1]$. If the coefficients $\widetilde{f}(h)$ decay with
order $h^{-\alpha}$ then $f_{\mathrm{per}}$ is $\alpha$-times
differentiable in the classical sense. However, this does not apply
to the sum over the odd coefficients. For instance, the cosine series for $x \mapsto x - \tfrac12$ is given by
\begin{align}\label{eq:x-expansion}
  -\frac{4}{\pi^2} \sum_{\substack{k=1\\[1mm] k \text{ odd}}}^{\infty} \frac{1}{k^2} \cos(\pi k x)
\end{align}
and hence the odd
coefficients converge with order $k^{-2}$ only, although $x$ is infinitely times differentiable.

Below we introduce a reproducing kernel Hilbert space based on cosine series, with the smoothness measured by the decay rate of the cosine coefficients. Although the smoothness of a cosine series measured by the differentiability of the series can be larger than the decay rate of the cosine coefficients suggests, the opposite can not happen. That is, we show that the reproducing kernel Hilbert space based on cosine series is embedded in the unanchored Sobolev space with the same value of the smoothness parameter. The case of smoothness $1$ provides an exception, since there the cosine series space and the unanchored Sobolev space coincide. Various reproducing kernel Hilbert spaces are introduced in Section~\ref{sec_rkhs} and their embeddings are studied in Section~\ref{sec_embeddings}.

In this paper we present two methods which allow us to achieve a higher convergence rate for smoother nonperiodic functions using lattice rules, namely:
\begin{enumerate}
 \item application of the tent transformation to the integration nodes;
 \item symmetrization of the integration nodes.
\end{enumerate}

\noindent
The {\it tent transformation}, $\phi : [0,1] \rightarrow [0,1]$,
$$
   \phi(x) := 1-|2x-1|
$$
is a Lebesgue measure preserving function. The idea of using this transformation in conjunction with a random shift for integration based on lattice rules comes from Hickernell \cite{H02} and was also used in \cite{CDPL} for digital nets. In contrast to these works, here we do not rely on a random element in our quadrature rules. We show that a tent transformed lattice rule achieves an integration error of order $N^{-\alpha + \delta}$, for any $\delta > 0$, for functions belonging to a certain reproducing kernel Hilbert space of cosine series with smoothness parameter $\alpha$. This result follows by showing that the worst-case error in the cosine series space for a tent-transformed lattice rule is the same as the worst-case error in a Korobov space of smooth periodic functions using lattice rules. Thus all the results for integration in Korobov spaces using lattice rules \cite{D04,DSWW,kuo,NC06,NC06b} also apply for integration in the cosine series space using tent-transformed lattice rules. In particular for smoothness $1$, this yields deterministic point sets for numerical integration in unanchored Sobolev spaces with the same tractability properties as for numerical integration in the Korobov space.

Furthermore, we also use symmetrized lattice rules. We show that these rules achieve the
optimal order of convergence for integration of sums of products of cosine series and Fourier series. We apply the
transformation $x \mapsto 1-x$ to each possible set of coordinates separately, so
that if we start off with $N$ points we get $O(2^{s-1} N)$ points (see Section~\ref{sec:symm}).
This symmetrization approach is also mentioned in \cite{kor,SJ,Zar72a} and
is one of the symmetry groups applied in the construction of cubature formulae, see, e.g., \cite{CH94,GM83}.
We prove that a lattice rule symmetrized this way achieves an
integration error of order $N^{-\alpha+\delta}$,
$\delta > 0$, for functions belonging to a certain reproducing kernel Hilbert space of cosine series and Fourier series with smoothness parameter $\alpha$. The advantage of symmetrized lattice rules is that functions of the form $\cos (\pi (2k-1) x)$, where $k$ is a nonnegative integer, are integrated exactly and hence only the smoothness of the periodic part determines the convergence rate. Note that the decay rate of the cosine series coefficients of the periodic part and of the Fourier series coefficients part coincides with the classical smoothness. Thus the problem with functions where the smoothness in terms of differentiability differs from the rate of decay of the cosine series is overcome using symmetrization. For instance, the function $x \mapsto x$ is integrated exactly using symmetrization. However, a disadvantage of the symmetrization compared to the tent transformation is that the number of function evaluations grows exponential in the dimension and therefore symmetrization is only useful in smaller dimensions.

For both methods, the rates of convergence we obtain are essentially optimal by an adaption of the lower bound of Bakhvalov \cite{bakh}, which is presented in Section~\ref{sec_lower_bound}.

In the next section we introduce four reproducing kernel Hilbert spaces, the unanchored Sobolev space, the Korobov space, the cosine series space and the sum of the cosine and Korobov space. Since the unanchored Sobolev space and the Korobov space are frequently studied in the literature, we study the relations among these four spaces in Section~\ref{sec_embeddings} to put our results into context. It is shown that the Korobov space and the cosine series space differ, but both are embedded in the sum of the cosine series and Korobov space, which is itself embedded in the unanchored Sobolev space. In Section~\ref{sec_num_int} we study numerical integration in the cosine series space using tent-transformed lattice rules and numerical integration in the sum of the cosine series and Korobov space using symmetrized lattice rules. Numerical results are presented in Section~\ref{sec_num_res} and a conclusion is presented in Section~\ref{sec_conclusion}.

We write $\mathbb{Z}$ for the set of integers, $\mathbb{N} := \{1, 2, \ldots\}$ for the set of positive integers and $\mathbb{N}_0 := \{0, 1, 2, \ldots\}$ for the set of nonnegative integers.
We also write $\mathbb{R}_+^s := \{x \in \mathbb{R}: x > 0\}^s$. Furthermore, for $s \in \NN$ we write $[s] := \{1,\ldots,s\}$.

\section{Reproducing kernel Hilbert spaces}\label{sec_rkhs}

In this section we introduce several reproducing kernel Hilbert spaces \cite{A50}. For a reproducing kernel $K:[0,1] \times [0,1]\to\mathbb{R}$ we denote by $\mathcal{H}(K)$ the corresponding reproducing kernel Hilbert space with inner product $\langle \cdot, \cdot \rangle_K$ and corresponding norm $\|f\|_K = \sqrt{\langle f, f \rangle_K}$. For any $y \in [0,1]$ we have $K(\cdot, y) \in \mathcal{H}(K)$ and we have the reproducing property
\begin{equation*}
f(y) = \langle f, K(\cdot, y)\rangle \quad \mbox{for all } y \in [0,1] \mbox{ and } f \in \mathcal{H}(K).
\end{equation*}
Further, the function $K$ is symmetric in its arguments and positive semi-definite.

For higher dimensions $s > 1$ we consider tensor product spaces. The reproducing kernel is in this case given by
\begin{equation*}
K_s(\bsx,\bsy) = \prod_{j=1}^s K(x_j,y_j),
\end{equation*}
where $\bsx=(x_1,\ldots,x_s)$ and $\bsy=(y_1,\ldots,y_s)$.
Again, the corresponding reproducing kernel Hilbert space is denoted by $\mathcal{H}(K_s)$, the corresponding inner product is denoted by $\langle \cdot, \cdot \rangle_{K_s}$ and the corresponding norm by $\|f\|_{K_s} = \sqrt{\langle f, f \rangle_{K_s}}$.

For further information on reproducing kernel Hilbert spaces we refer to \cite{A50} (or \cite[Chapter~2]{DP} for reproducing kernel Hilbert spaces in the context of numerical integration).

\subsection{The unanchored Sobolev space}

The {\it unanchored Sobolev space} is a reproducing kernel Hilbert space $\mathcal{H}(K^{\sob}_{1,\gamma})$ with reproducing kernel $K^{\sob}_{1,\gamma}:[0,1]\times [0,1] \to \mathbb{R}$ given by
\begin{equation*}
K^{\sob}_{1,\gamma}(x,y) := 1 + \gamma B_1(x) B_1(y) + \gamma
\frac{B_2(|x-y|)}{2},
\end{equation*}
where $B_1(z) = z-1/2$ and $B_2(z) = z^2-z+1/6$ are Bernoulli
polynomials and $\gamma > 0$ is a real number. The inner product in this space
is given by
\begin{equation*}
\langle f, g \rangle_{K^{\sob}_{1,\gamma}} = \int_0^1 f(x)
\,\mathrm{d} x \int_0^1 g(x) \,\mathrm{d} x + \frac{1}{\gamma} \int_{0}^1
f'(x) g'(x) \,\mathrm{d} x.
\end{equation*}
In higher dimensions $s>1$ we consider the tensor product space
$\mathcal{H}(K^{\sob}_{1,\gamma_1}) \otimes \cdots \otimes
\mathcal{H}(K^{\sob}_{1,\gamma_s})$. The reproducing kernel is
in this case given by
\begin{equation*}
K^{\sob}_{1,\bsgamma,s}(\bsx,\bsy) := \prod_{j=1}^s
K^{\sob}_{1,\gamma_j}(x_j,y_j),
\end{equation*}
where $\bsx = (x_1,\ldots, x_s)$, $\bsy = (y_1,\ldots, y_s)$ and
$\bsgamma = (\gamma_1,\ldots, \gamma_s)\in \mathbb{R}_+^s$.

The reproducing kernel Hilbert space
$\mathcal{H}(K^{\sob}_{1,\gamma})$ can be generalized to higher
order smoothness. For $\alpha \in \mathbb{N}$, consider the
reproducing kernel $K^{\sob}_{\alpha,\gamma}:[0,1]\times [0,1] \to
\mathbb{R}$ given by
\begin{equation*}
K^{\sob}_{\alpha,\gamma}(x,y) := 1 + \gamma \sum_{\tau=1}^\alpha \frac{B_\tau(x) B_\tau(y)}{(\tau!)^2} - (-1)^\alpha \gamma \frac{B_{2\alpha}(|x-y|)}{(2\alpha)!},
\end{equation*}
where $B_\tau$ is the Bernoulli polynomial of order $\tau$. The inner product for this space is given by
\begin{multline*}
\langle f, g \rangle_{K^{\sob}_{\alpha,\gamma}} = \int_0^1 f(x) \,\mathrm{d} x \int_0^1 g(x) \,\mathrm{d} x \\*
+ \frac{1}{\gamma} \sum_{\tau=1}^{\alpha-1} \int_0^1 f^{(\tau)}(x)\,\mathrm{d} x \int_0^1 g^{(\tau)}(x) \,\mathrm{d} x
+ \frac{1}{\gamma} \int_0^1 f^{(\alpha)}(x) \, g^{(\alpha)}(x) \,\mathrm{d} x.
\end{multline*}
To obtain reproducing kernel Hilbert spaces for the domain
$[0,1]^s$, we consider again the tensor product of the
one-dimensional spaces. This space has the reproducing kernel
\begin{equation*}
K^{\sob}_{\alpha,\bsgamma,s}(\bsx,\bsy) := \prod_{j=1}^s
K^{\sob}_{\alpha,\gamma_j}(x_j,y_j).
\end{equation*}

Quasi-Monte Carlo rules in $\mathcal{H}(K^{\sob}_{\alpha,\bsgamma,s})$ which yield the optimal rate of convergence of order $N^{-\alpha+\delta}$ for any $\delta > 0$ where studied in \cite{D08}.

\subsection{The Korobov space}

For $\alpha > 1/2$, $h \in \mathbb{Z}$ and $\gamma > 0$ we define
\begin{equation}\label{defr}
r_{\alpha,\gamma}(h) := \left\{\begin{array}{ll} 1 & \mbox{if } h =
0, \\ \gamma |h|^{-2\alpha} & \mbox{if } h \neq 0. \end{array} \right.
\end{equation}
For $\bsh = (h_1,\ldots, h_s) \in \mathbb{Z}^s$ and $\bsgamma = (\gamma_1,\ldots,\gamma_s) \in \RR_+^s$, we set
\begin{equation*}
r_{\alpha,\bsgamma,s}(\bsh) := \prod_{j = 1}^s r_{\alpha,\gamma_j}(h_j).
\end{equation*}

The {\it Korobov space} is a reproducing kernel Hilbert space of Fourier
series. The reproducing kernel for this space is given by
\begin{equation*}
K^{\kor}_{\alpha,\gamma}(x,y) := \sum_{h \in \mathbb{Z}} r_{\alpha,\gamma}(h) \, \mathrm{e}^{2\pi \mathrm{i} h (x-y)},
\end{equation*}
and in higher dimensions $s>1$ by
\begin{equation*}
K^{\kor}_{\alpha,\bsgamma,s}(\bsx,\bsy) := \prod_{j=1}^s K^{\kor}_{\alpha,\gamma_j}(x_j,y_j) = \sum_{\bsh \in
\mathbb{Z}^s} r_{\alpha,\bsgamma,s}(\bsh) \, \mathrm{e}^{2\pi
\mathrm{i} \bsh\cdot (\bsx-\bsy)},
\end{equation*}
where the ``$\cdot$'' denotes the usual inner product in $\RR^s$.
Let the Fourier coefficient for a function $f:[0,1]^s \to
\mathbb{R}$ be given by
\begin{equation*}
\widehat{f}(\bsh) := \int_{[0,1]^s} f(\bsx) \, \mathrm{e}^{-2\pi
\mathrm{i} \bsh \cdot \bsx} \,\mathrm{d} \bsx.
\end{equation*}
Then the inner product in the reproducing kernel Hilbert space
$\mathcal{H}(K^{\kor}_{\alpha,\bsgamma,s})$ is given by
\begin{equation*}
\langle f, g \rangle_{K^{\kor}_{\alpha,\bsgamma,s}} = \sum_{\bsh \in
\mathbb{Z}^s} \widehat{f}(\bsh) \, \overline{\widehat{g}(\bsh)} \,
r^{-1}_{\alpha,\bsgamma,s}(\bsh).
\end{equation*}
The corresponding norm is defined by
$\|f\|_{K^{\kor}_{\alpha,\bsgamma,s}} = \sqrt{\langle f, f
\rangle_{K^{\kor}_{\alpha, \bsgamma,s}}}$.

\subsection{The half-period cosine space}

The {\it half-period cosine space} is a reproducing kernel Hilbert space of (half-period) cosine series with reproducing kernel
\begin{equation*}
K^{\cos}_{\alpha,\gamma}(x,y) := 1+\sum_{k=1}^\infty
r_{\alpha,\gamma}(k) \sqrt{2} \cos(\pi k x) \sqrt{2} \cos(\pi k y),
\end{equation*}
where $r_{\alpha,\gamma}$ is defined as in \eqref{defr} and where $\alpha > 1/2$ and $\gamma > 0$.
The inner product is given by
\begin{equation*}
\langle f, g \rangle_{K^{\cos}_{\alpha,\gamma}} = \sum_{k=0}^\infty
\widetilde{f}_{\cos}(k) \, \widetilde{g}_{\cos}(k) \, r^{-1}_{\alpha,\gamma}(k),
\end{equation*}
where $\widetilde{f}_{\cos}$ and $\widetilde{g}_{\cos}$ are the cosine coefficients of $f$ and $g$, respectively, as defined in \eqref{eq:that}.

We can generalize the reproducing kernel Hilbert space
$\mathcal{H}(K^{\cos}_{\alpha,\gamma})$ to the domain $[0,1]^s$ by
setting
\begin{equation*}
K^{\cos}_{\alpha,\bsgamma,s}(\bsx,\bsy) := \prod_{j=1}^s
K^{\cos}_{\alpha,\gamma_j}(x_j,y_j),
\end{equation*}
where $\bsgamma = (\gamma_1,\ldots, \gamma_s) \in \mathbb{R}_+^s$ and $\bsx = (x_1,\ldots, x_s),
\bsy=(y_1,\ldots, y_s) \in [0,1]^s$. The inner product is then given
by
\begin{equation*}
\langle f, g \rangle_{K^{\cos}_{\alpha,\bsgamma,s}} = \sum_{\bsk \in
\mathbb{N}_0^s} \widetilde{f}_{\cos}(\bsk) \, \widetilde{g}_{\cos}(\bsk) \,
r_{\alpha,\bsgamma,s}^{-1}(\bsk),
\end{equation*}
where the multi-dimensional cosine coefficients for a function $f:[0,1]^s \to\mathbb{R}$ are given by
\begin{equation*}
\widetilde{f}_{\cos}(\bsk) := \int_{[0,1]^s} f(\bsx) \, 2^{|\bsk|_0/2} \prod_{j=1}^s \cos(\pi k_j x_j) \,\mathrm{d} \bsx,
\end{equation*}
where for $\bsk=(k_1,\ldots,k_s) \in \NN_0^s$ we define $|\bsk|_0 := |\{j \in [s]\, : \, k_j \not=0\}|$ to be the number of nonzero components in $\bsk$.
The corresponding norm is defined by
$\|f\|_{K^{\cos}_{\alpha,\bsgamma,s}} = \sqrt{\langle f, f\rangle_{K^{\cos}_{\alpha,\bsgamma,s}}}$, in particular we have
\begin{align*}
  \|f\|_{K^{\cos}_{\alpha,\bsgamma,s}}^2
  &=
  \sum_{\bsk \in \mathbb{N}_0^s} \frac{|\widetilde{f}_{\cos}(\bsk)|^2}{r_{\alpha,\bsgamma,s}(\bsk)}
  =
  \sum_{\bsh \in \mathbb{Z}^s} 2^{-|\bsh|_0} \frac{|\widetilde{f}_{\cos}(|\bsh|)|^2}{r_{\alpha,\bsgamma,s}(\bsh)}
  ,
\end{align*}
where $|\bsh| = (|h_1|, \ldots, |h_d|)$.

\subsection{The sum of the Korobov space and the half-period cosine space}

In this section we introduce the kernel
\begin{eqnarray*}
K^{\kor + \cos}_{\alpha,\gamma}(x,y)
& := &
\frac{1}{2} \left( K^{\kor}_{\alpha,\gamma}(x,y) + K^{\cos}_{\alpha,\gamma}(x,y) \right) \\
& \hphantom{:}= &
\frac{1}{2} \sum_{h\in \mathbb{Z}} r_{\alpha,\gamma}(h) \, \mathrm{e}^{2\pi \mathrm{i} h (x-y) }
+ \frac{1}{2} + \sum_{k=1}^\infty r_{\alpha,\gamma}(k) \cos(\pi k x) \cos(\pi k y) \\
& \hphantom{:}= &
1 + \gamma \sum_{k=1}^\infty k^{-2\alpha} \left( \cos(2\pi k (x-y)) + \cos(\pi k x) \cos(\pi k y) \right).
\end{eqnarray*}
The space resulting from the sum of kernels is studied in \cite[Part~I, Section~6]{A50}.
The norm in the reproducing kernel Hilbert space $\mathcal{H}(K^{\kor + \cos}_{\alpha,\gamma})$ is then defined by
\begin{equation*}
\|f\|^2_{K^{\kor+\cos}_{\alpha,\gamma}} =
\min_{f = f_{\kor} + f_{\cos}} 2 \left( \|f_{\kor}\|^2_{K^{\kor}_{\alpha,\gamma}} + \|f_{\cos}\|^2_{K^{\cos}_{\alpha,\gamma}} \right),
\end{equation*}
where the minimum is taken over all functions $f_{\kor} \in \mathcal{H}(K^{\kor}_{\alpha,\gamma})$ and $f_{\cos} \in \mathcal{H}(K^{\cos}_{\alpha,\gamma})$ such that $f = f_{\kor} + f_{\cos}$.

For dimensions $s > 1$ we define the reproducing kernel by
\begin{equation*}
K^{\kor + \cos}_{\alpha,\bsgamma,s}(\bsx,\bsy) := \prod_{j=1}^s \frac{1}{2} \left(K^{\kor}_{\alpha,\gamma_j}(x_j,y_j) + K^{\cos}_{\alpha,\gamma_j}(x_j,y_j) \right).
\end{equation*}
Thus the space $\mathcal{H}(K^{\kor+ \cos}_{\alpha,\bsgamma,s})$ is the tensor product $\mathcal{H}(K^{\kor+\cos}_{\alpha,\gamma_1}) \otimes \cdots \otimes \mathcal{H}(K^{\kor+\cos}_{\alpha,\gamma_s})$. For $\uu \subseteq [s]$ we define
\begin{equation*}
K^{\kor + \cos}_{\alpha,\bsgamma,s,\uu}(\bsx,\bsy) := 2^{-s} \prod_{j \in \uu} K^{\kor}_{\alpha,\gamma_j}(x_j,y_j) \prod_{j \in [s]\setminus \uu}  K^{\cos}_{\alpha,\gamma_j}(x_j,y_j),
\end{equation*}
where as usual an empty product is considered to be one. This is a reproducing kernel for the space
$$
  \mathcal{H}(K^{\kor+\cos}_{\alpha,\bsgamma,s,\uu})
  =
  \left(\bigotimes_{j\in \uu \vphantom{[s]\setminus \uu}} \mathcal{H}(\tfrac12 K^{\kor}_{\alpha,\gamma_j}) \right)
  \otimes
  \left(\bigotimes_{j\in [s]\setminus \uu} \mathcal{H}(\tfrac12 K^{\cos}_{\alpha,\gamma_j})\right)
$$
with the inner product
$$
  \langle f,g\rangle_{K^{\kor + \cos}_{\alpha,\bsgamma,s,\uu}}
  =
  2^{s} \sum_{\bsh_{\uu}\in \ZZ^{|\uu|}} \sum_{\bsk_{[s]\setminus \uu} \in \NN_0^{s-|\uu|}}
  \widetilde{f}_{\uu,\kor+\cos}(\bsh_{\uu},\bsk_{[s]\setminus \uu}) \,
  \overline{ \widetilde{g}_{\uu,\kor+\cos}(\bsh_{\uu},\bsk_{[s]\setminus \uu})} \,
  r_{\alpha,\gamma,s}^{-1}(\bsh_{\uu},\bsk_{[s]\setminus \uu}),
$$
where
$$
  \widetilde{f}_{\uu,\kor+\cos}(\bsh_{\uu},\bsk_{[s]\setminus \uu})
  :=
  \int_{[0,1]^s} f(\bsx) \, (\sqrt{2})^{|\bsk_{[s]\setminus \uu}|_0} \prod_{j \in \uu} {\rm e}^{-2 \pi \mathrm{i} h_j x_j} \prod_{j \in [s]\setminus \uu} \cos(\pi k_j x_j)
  \, \mathrm{d}\bsx
  .
$$
Clearly we have that
$$
  K^{\kor + \cos}_{\alpha,\bsgamma,s}(\bsx,\bsy)
  =
  \sum_{\uu \subseteq [s]} K^{\kor + \cos}_{\alpha,\bsgamma,s,\uu}(\bsx,\bsy)
  .
$$

\section{Embeddings}\label{sec_embeddings}

In this section we investigate the relationships between the spaces introduced above.

For two reproducing kernel Hilbert spaces $\mathcal{H}(K_1)$ and $\mathcal{H}(K_2)$ we say that $\mathcal{H}(K_1)$ is {\it continuously embedded} in $\mathcal{H}(K_2)$ if $$\mathcal{H}(K_1) \subseteq \mathcal{H}(K_2)$$ and if $$\|f\|_{K_2} \le C \|f\|_{K_1} \quad \mbox{for all } f \in \mathcal{H}(K_1)$$ for some constant $C>0$ independent of $f$. We write $$\mathcal{H}(K_1) \hookrightarrow \mathcal{H}(K_2)$$ in this case.

On the other hand it is possible that $\mathcal{H}(K_1)$ is not a subset of $\mathcal{H}(K_2)$, i.e., there is a function in $\mathcal{H}(K_1)$ which is not in $\mathcal{H}(K_2)$. In this case we write $\mathcal{H}(K_1) \not\subset \mathcal{H}(K_2)$.

If $\mathcal{H}(K_1) \hookrightarrow \mathcal{H}(K_2)$ and $\mathcal{H}(K_2) \hookrightarrow \mathcal{H}(K_1)$ we write
\begin{equation*}
\mathcal{H}(K_1) \leftrightharpoons \mathcal{H}(K_2).
\end{equation*}

\subsection{The Korobov space and the unanchored Sobolev space}

It is well known, see, e.g., \cite[Appendix~A]{NW08}, that the Korobov space is continuously embedded in the unanchored Sobolev space
\begin{equation*}
\mathcal{H}(K^{\kor}_{\alpha,\bsgamma,s}) \hookrightarrow \mathcal{H}(K^{\sob}_{\alpha,\bsgamma,s}).
\end{equation*}
Conversely, as is also well known, for instance the function $g:[0,1]^s \rightarrow \RR$, $g(\bsx)=x_1$ is in $\mathcal{H}(K^{\sob}_{\alpha,\bsgamma,s})$ for all $\alpha \in \mathbb{N}$, but not in $\mathcal{H}(K^{\kor}_{\alpha,\bsgamma,s})$, since $g$ is not periodic. Thus $$\mathcal{H}(K^{\sob}_{\alpha,\bsgamma,s}) \not\subset \mathcal{H}(K^{\kor}_{\alpha,\bsgamma,s}).$$
We note that for $f \in \mathcal{H}(K^{\kor}_{\alpha,\bsgamma,s})$, $\alpha \in \NN$, we have that $\|f\|_{K^{\sob}_{\alpha,\bsgamma,s}} = \|f\|_{K^{\sob}_{\alpha,\bsgamma(2\pi)^{-2\alpha},s}}$ where $\bsgamma(2\pi)^{-2\alpha}$ denotes the rescaled sequence $(\gamma_1 (2\pi)^{-2\alpha}, \ldots, \gamma_s (2\pi)^{-2\alpha})$.

\subsection{The half-period cosine space and the unanchored Sobolev space}

We now consider the half-period cosine space and the unanchored Sobolev space. For $\alpha=1$ there is a peculiarity.
\begin{lemma}
  We have
  \begin{equation*}
    K^{\sob}_{1,\gamma}(x,y)
    =
    1 + \frac{\gamma}{\pi^2} \sum_{k=1}^\infty \frac{1}{k^2} 2 \cos(k \pi x) \cos(k \pi y)
    =
    K^{\cos}_{1,\gamma \pi^{-2}}(x,y).
  \end{equation*}
\end{lemma}

For completeness we include a short proof.
\begin{proof}
In the following we calculate the cosine coefficients of $K^{\sob}_{1,\gamma}$. It is easy to check that we have $\int_0^1 \int_0^1 K^{\sob}_{1,\gamma}(x,y) \cos(\pi n x) \cos(\pi m y) \rd{x} \rd{y} = 0$ for $(n,m) \in \{(0,k): k>0\} \cup \{(k,0): k>0\}$. Further we have $\int_0^1 \int_0^1 K^{\sob}_{1,\gamma}(x,y) \rd{x} \rd{y} = 1$.

We have
    \begin{align*}
    B_2(|x-y|)  = & B_2(\{x-y\}) = \frac{1}{2\pi^2} \sum_{k\in \mathbb{Z}\setminus \{0\}} \frac{\mathrm{e}^{2\pi \mathrm{i} k (x-y)}}{k^2} \\  = & \sum_{k=1}^\infty \frac{\cos (2\pi k x) \cos(2\pi k y)}{\pi^2 k^2} + \sum_{k=1}^\infty \frac{\sin(2\pi kx) \sin(2\pi ky)}{\pi^2 k^2}.
    \end{align*}
Using \eqref{eq:x-expansion} we obtain
    \begin{align*}
    (x-\tfrac{1}{2})(y-\tfrac{1}{2}) + \frac{B_2(|x-y|)}{2} = & \sum_{k,l=1}^\infty \frac{16}{\pi^4 (2k-1)^2(2l-1)^2} \cos(\pi(2k-1)x) \cos(\pi(2l-1)y) \\ & + \sum_{k=1}^\infty \frac{\cos (2\pi k x) \cos(2\pi k y)}{2\pi^2 k^2} + \sum_{k=1}^\infty \frac{\sin(2\pi kx) \sin(2\pi ky)}{2\pi^2 k^2}.
    \end{align*}
This immediately implies that
\begin{equation*}
\int_0^1 \int_0^1 K^{\sob}_{1,\gamma} \sqrt{2} \cos(\pi m x) \cos(\pi n y) \rd{x} \rd{y} = 0
\end{equation*}
if $m$ is even and $n$ is odd, or $m$ is odd and $n$ is even, or $m,n$ are even with $m \neq n$. If $m=n=k$ for even $k > 0$, we obtain
  \begin{equation*}
\int_0^1 \int_0^1  K^{\sob}_{1,\gamma}(x,y) \sqrt{2} \cos(\pi k x) \sqrt{2}\cos(\pi k y) \rd{x} \rd{y} = \frac{\gamma}{\pi^2 k^2}.
  \end{equation*}

Now let $m,n > 0$ be odd. We have $\int_0^1 \sin(2\pi k x) \cos(\pi m x) \rd{x} = \frac{2k(1-(-1)^m)}{\pi(4k^2-m^2)}$ and therefore
\begin{align*}
\int_0^1 \int_0^1 K^{\sob}_{1,\gamma}(x,y) \sqrt{2} \cos(\pi m x) \sqrt{2} \cos(\pi n y) \rd{x} \rd{y} = & \frac{8}{\pi^4 m^2 n^2} + \frac{16}{\pi^4} \sum_{k=1}^\infty \frac{1}{(4k^2-m^2)(4k^2-n^2)}.
\end{align*}
For $m \neq n$ we have $\sum_{k=1}^\infty \frac{1}{(4k^2-m^2)(4k^2-n^2)} = -\frac{1}{m^2n^2}$ and further $\sum_{k=1}^\infty \frac{1}{(4k^2-m^2)^2} = \frac{\pi^2 m^2-8}{16 m^4}$. Thus we obtain
\begin{align*}
\int_0^1 \int_0^1 K^{\sob}_{1,\gamma}(x,y) \sqrt{2} \cos(\pi m x) \sqrt{2} \cos(\pi n y) \rd{x} \rd{y}  = \left\{\begin{array}{rl} 0 & \mbox{if } m \neq n, \\ \frac{\gamma}{(\pi m)^{2}}  & \mbox{if } m = n. \end{array} \right.
\end{align*}
Note that the cosine series for $K^{\sob}_{1,\gamma}$ converges absolutely. Since the function $K^{\sob}_{1,\gamma}$ is continuous, the cosine series converges to the function pointwise. This completes the proof.
\end{proof}

The above lemma and Mercer's theorem also yield the eigenfunctions of the operator $$T(g)(y) = \int_0^1 K^{\sob}_{1,\gamma}(x,y) g(x) \rd{x}.$$ These are $1, \sqrt{2} \cos(\pi x), \sqrt{2} \cos(2\pi x), \sqrt{2} \cos(3\pi x),\ldots$ and the corresponding eigenvalues are $1, \pi^{-2}, (\pi 2)^{-2}, (\pi 3)^{-2}, \ldots$. 

\begin{remark}
An analoguous result for a slightly different reproducing kernel Hilbert space was established in \cite{WW09}. For this space one obtains the same set of eigenfunctions. For the anchored Sobolev space the eigenfunctions are slightly different and have been found in \cite{WW99}.
\end{remark}

We thus find
\begin{equation*}
\mathcal{H}(K^{\sob}_{1,\gamma}) = \mathcal{H}(K^{\cos}_{1,\gamma\pi^{-2}}) \quad \mbox{and} \quad \|f\|_{K^{\sob}_{1,\gamma}} = \|f\|_{K^{\cos}_{1,\gamma \pi^{-2}}} \quad \mbox{for all } f \in \mathcal{H}(K^{\sob}_{1,\gamma}).
\end{equation*}
The same also applies for the higher dimensional tensor product space
\begin{equation*}
\mathcal{H}(K^{\sob}_{1,\bsgamma,s}) = \mathcal{H}(K^{\cos}_{1,\bsgamma\pi^{-2}, s}) \quad \mbox{and} \quad \|f\|_{K^{\sob}_{1,\bsgamma,s}} = \|f\|_{K^{\cos}_{1,\bsgamma \pi^{-2}, s}} \quad \mbox{for all } f \in \mathcal{H}(K^{\sob}_{1,\bsgamma, s}),
\end{equation*}
where $\bsgamma \pi^{-2}$ denotes the sequence $(\gamma_1 \pi^{-2},\ldots,\gamma_s \pi^{-2})$. Thus
\begin{equation*}
\mathcal{H}(K^{\sob}_{1,\bsgamma,s}) \leftrightharpoons \mathcal{H}(K^{\cos}_{1,\bsgamma,s}).
\end{equation*}

Now consider $\alpha > 1$. The function $x \mapsto x$ belongs to $\mathcal{H}(K^{\sob}_{\alpha,\gamma})$ for all $\alpha \in \mathbb{N}$. On the other hand we have
\begin{equation*}
x = \frac{1}{2} - \frac{4}{\pi^2} \sum_{\substack{k=1 \\[1mm] \text{$k$ odd}}}^\infty \frac{\cos(\pi k x)}{k^2} .
\end{equation*}
Hence the function $x\mapsto x$ is not in $\mathcal{H}(K^{\cos}_{\alpha,\gamma})$ for $\alpha \ge 3/2$ and therefore
\begin{equation*}
\mathcal{H}(K^{\sob}_{\alpha,\gamma}) \not\subset \mathcal{H}(K^{\cos}_{\alpha,\gamma}) \quad \mbox{for } \alpha \in \mathbb{N}, \alpha \ge 2.
\end{equation*}

Conversely, let $\alpha \in \mathbb{N}$, $\alpha \ge 2$. Let $f \in \mathcal{H}(K^{\cos}_{\alpha,\gamma})$ be given by
\begin{equation*}
f(x) = \widetilde{f}_{\cos}(0) +  \sum_{k=1}^\infty \widetilde{f}_{\cos}(k) \sqrt{2} \cos(\pi k x)
\end{equation*}
with
\begin{equation*}
\|f\|^2_{K^{\cos}_{\alpha,\gamma}} = |\widetilde{f}_{\cos}(0)|^2 + \frac{1}{\gamma} \sum_{k=1}^\infty |\widetilde{f}_{\cos}(k)|^2 |k|^{2\alpha} < \infty.
\end{equation*}
Then for $1 \le \tau \le \alpha$ we have
\begin{equation*}
f^{(\tau)}(x) = \sum_{k=1}^\infty \widetilde{f}_{\cos}(k) (-1)^{\lceil \tau/2 \rceil} (k\pi)^\tau \sqrt{2} \, \phi_\tau(\pi k x),
\end{equation*}
where $\phi_\tau(z) = \cos(z)$ for $\tau$ even and $\phi_\tau(z) = \sin(z)$ for $\tau$ odd and where for a real number $x$, $\lceil x \rceil$ denotes the smallest integer bigger or equal to $x$. Thus
\begin{equation*}
\frac{1}{\gamma} \left|\int_0^1 f^{(\tau)}(x) \,\mathrm{d} x \right|^2 \le \frac{1}{\gamma} \int_0^1 |f^{(\tau)}(x)|^2 \,\mathrm{d} x = \frac{\pi^{2\tau}}{\gamma}  \sum_{k=1}^\infty |\widetilde{f}_{\cos}(k)|^2 k^{2\tau} \le \pi^{2\tau} \|f\|^2_{K^{\cos}_{\alpha,\gamma}}.
\end{equation*}
Thus
\begin{equation*}
\|f\|_{K^{\sob}_{\alpha,\gamma}} \le  \left(\sum_{\tau=0}^\alpha \pi^{2\tau}\right)^{1/2} \|f\|_{K^{\cos}_{\alpha,\gamma}} = \left(\frac{\pi^{2(\alpha+1)}-1}{\pi^2-1}\right)^{1/2} \|f\|_{K^{\cos}_{\alpha,\gamma}}.
\end{equation*}
Thus we have
\begin{equation*}
\mathcal{H}(K^{\cos}_{\alpha,\gamma}) \hookrightarrow \mathcal{H}(K^{\sob}_{\alpha,\gamma}).
\end{equation*}
This result can be generalized to the tensor product space, thus
\begin{equation*}
\mathcal{H}(K^{\cos}_{\alpha,\bsgamma,s}) \hookrightarrow \mathcal{H}(K^{\sob}_{\alpha,\bsgamma,s}).
\end{equation*}

\subsection{The half-period cosine space and the Korobov space}

For $\alpha=1$ the embedding results for the half-period cosine space and the Korobov space follow from the previous two subsections.

Let now $\alpha > 1$. Let
\begin{equation*}
f(x) = \sin(2\pi x).
\end{equation*}
Then $f \in \mathcal{H}(K^{\kor}_{\alpha,\gamma})$ for all $\alpha > 1/2$. On the other hand we have for $k \in \mathbb{N}$
\begin{equation*}
\widetilde{f}_{\cos}(k) = \int_0^1 \sin(2\pi x) \sqrt{2} \cos(\pi k x) \,\mathrm{d} x = \left\{\begin{array}{ll} \frac{4\sqrt{2}}{\pi(4-k^2)} & \mbox{if } k \mbox{ is odd}, \\ 0 & \mbox{otherwise}. \end{array} \right.
\end{equation*}
Thus
\begin{equation*}
\|f\|_{K^{\cos}_{\alpha,\gamma}}^2 = \frac{1}{\gamma}\sum_{\satop{k=1}{k \ {\rm odd}}}^\infty |\widetilde{f}_{\cos}(k)|^2 k^{2\alpha} = \frac{1}{\gamma}\frac{32}{\pi^2} \sum_{k=1}^\infty \frac{(2k-1)^{2\alpha}}{((2k-1)^2-4)^2}.
\end{equation*}
Thus we have $\|f\|_{K^{\cos}_{\alpha,\gamma}} = \infty$ for $\alpha \ge 3/2$. Thus
\begin{equation*}
\mathcal{H}(K^{\kor}_{\alpha,\gamma}) \not\subset \mathcal{H}(K^{\cos}_{\alpha,\gamma}) \quad \mbox{for } \alpha \ge 3/2.
\end{equation*}

Conversely, let now
\begin{equation*}
f(x) = \cos(\pi x).
\end{equation*}
Then $f \in \mathcal{H}(K^{\cos}_{\alpha,\gamma})$ for all $\alpha > 1/2$. On the other hand we have for $h \in \mathbb{Z}$
\begin{equation*}
\widehat{f}_{\kor}(h) = \int_0^1 \cos(\pi x) \, \mathrm{e}^{-2\pi \mathrm{i} h x} \, \mathrm{d} x = \frac{4 \mathrm{i} h}{\pi (1 - 4h^2)}.
\end{equation*}
For any $\alpha \ge 1/2$ we have
\begin{equation*}
\|f\|^2_{K^{\kor}_{\alpha,\gamma}} = \frac{1}{\gamma}\frac{16}{\pi^2} \sum_{h \in \mathbb{Z} \setminus \{0\}} |h|^{2\alpha} \frac{h^2}{(4h^2-1)^2} = \infty,
\end{equation*}
which implies that $f \notin \mathcal{H}(K^{\kor}_{\alpha,\gamma})$.
As we need $\alpha > 1/2$ we therefore have
\begin{equation*}
\mathcal{H}(K^{\cos}_{\alpha,\gamma}) \not\subset \mathcal{H}(K^{\kor}_{\alpha,\gamma})
.
\end{equation*}

\subsection{Embeddings of the sum of the Korobov and half-period cosine space}

Since $K^{\kor+\cos}_{\alpha,\gamma} = \frac{1}{2} K^{\kor}_{\alpha,\gamma} + \frac{1}{2} K^{\cos}_{\alpha,\gamma}$ we obtain from results from \cite{A50} that
\begin{align*}
\mathcal{H}(K^{\kor}_{\alpha,\gamma}) & \hookrightarrow \mathcal{H}(K^{\kor+\cos}_{\alpha,\gamma}), \\ \mathcal{H}(K^{\cos}_{\alpha,\gamma}) & \hookrightarrow \mathcal{H}(K^{\kor+\cos}_{\alpha,\gamma}).
\end{align*}
On the other hand, as it was shown above, we have $\mathcal{H}(K^{\kor}_{\alpha,\gamma}), \mathcal{H}(K^{\cos}_{\alpha,\gamma}) \hookrightarrow \mathcal{H}(K^{\sob}_{\alpha,\gamma})$, thus
\begin{equation*}
\mathcal{H}(K^{\kor+\cos}_{\alpha,\gamma}) \hookrightarrow \mathcal{H}(K^{\sob}_{\alpha,\gamma}).
\end{equation*}

We have seen above that for $\alpha \ge 1/2$ we have $\mathcal{H}(K^{\cos}_{\alpha,\gamma}) \not\subset \mathcal{H}(K^{\kor}_{\alpha,\gamma})$, thus
\begin{equation*}
\mathcal{H}(K^{\kor+\cos}_{\alpha,\gamma}) \not\subset \mathcal{H}(K^{\kor}_{\alpha,\gamma})
\end{equation*}
and for $\alpha \ge 3/2$ we have $\mathcal{H}(K^{\kor}_{\alpha,\gamma}) \not\subset \mathcal{H}(K^{\cos}_{\alpha,\gamma})$ and therefore
\begin{equation*}
\mathcal{H}(K^{\kor + \cos}_{\alpha,\gamma}) \not\subset \mathcal{H}(K^{\cos}_{\alpha,\gamma})
\quad \mbox{for } \alpha \ge 3/2.
\end{equation*}

For $\alpha = 1$ we have $\mathcal{H}(K^{\sob}_{1,\gamma}) \leftrightharpoons \mathcal{H}(K^{\cos}_{1,\gamma})$ and therefore
\begin{equation*}
\mathcal{H}(K^{\sob}_{1,\gamma}) \leftrightharpoons \mathcal{H}(K^{\kor+\cos}_{1,\gamma}).
\end{equation*}


\subsection{Summary of embeddings}

We summarize the obtained embedding results in the following theorem.

\begin{theorem}
For $\alpha = 1$ we have
\begin{equation*}
\mathcal{H}(K^{\kor}_{1,\bsgamma,s}) \hookrightarrow \mathcal{H}(K^{\sob}_{1,\bsgamma,s}) \leftrightharpoons \mathcal{H}(K^{\cos}_{1,\bsgamma,s}) \leftrightharpoons \mathcal{H}(K^{\kor+\cos}_{1,\bsgamma,s})
\end{equation*}
and
\begin{equation*}
\mathcal{H}(K^{\sob}_{1,\bsgamma,s}),  \mathcal{H}(K^{\cos}_{1,\bsgamma,s}), \mathcal{H}(K^{\kor+\cos}_{1,\bsgamma,s}) \not \subset \mathcal{H}(K^{\kor}_{1,\bsgamma,s}).
\end{equation*}
For $\alpha \in \mathbb{N}$ with $\alpha > 1$ we have
\begin{equation*}
\mathcal{H}(K^{\kor}_{\alpha,\bsgamma,s}), \mathcal{H}(K^{\cos}_{\alpha,\bsgamma,s}) \hookrightarrow \mathcal{H}(K^{\kor+\cos}_{\alpha,\bsgamma,s}) \hookrightarrow \mathcal{H}(K^{\sob}_{\alpha,\bsgamma,s})
\end{equation*}
and for $\alpha \in \RR$ we have
\begin{align*}
  \mathcal{H}(K^{\cos}_{\alpha,\bsgamma,s})
  &\not\subset
  \mathcal{H}(K^{\kor}_{\alpha,\bsgamma,s})
  ,
  \\
  \mathcal{H}(K^{\kor+\cos}_{\alpha,\bsgamma,s})
  &\not\subset
  \mathcal{H}(K^{\kor}_{\alpha,\bsgamma,s})
  ,
  \\
  \mathcal{H}(K^{\kor+\cos}_{\alpha,\bsgamma,s})
  &\not\subset
  \mathcal{H}(K^{\cos}_{\alpha,\bsgamma,s})
  \quad \mbox{for } \alpha \ge 3/2
  ,
  \\
  \mathcal{H}(K^{\kor}_{\alpha,\bsgamma,s})
  &\not\subset \mathcal{H}(K^{\cos}_{\alpha,\bsgamma,s})
  \quad \mbox{for } \alpha \ge 3/2
  .
\end{align*}
\end{theorem}

We do not know whether  $\mathcal{H}(K^{\sob}_{\alpha,\bsgamma,s})$ differs from $\mathcal{H}(K^{\kor+\cos}_{\alpha,\bsgamma,s})$ for $\alpha > 1$.

\section{Numerical integration}\label{sec_num_int}

We now study the worst-case error for QMC
integration. As quality measure for the QMC algorithm we use the worst-case integration error.

Let $P = \{\bsx_0,\ldots, \bsx_{N-1}\}$ and let $\mathcal{H}(K)$ be an arbitrary reproducing kernel Hilbert space with reproducing kernel $K$ and norm $\|\cdot\|_K$. Then the worst-case error for QMC integration in $\mathcal{H}(K)$ using the point set $P$ is defined as
\begin{equation*}
e(\mathcal{H}(K); P) := \sup_{\satop{f \in \mathcal{H}(K)}{\|f\|_{K}
\le 1}} \left|\int_{[0,1]^s} f(\bsx)\rd \bsx - \frac{1}{N}
\sum_{n=0}^{N-1} f(\bsx_n) \right|,
\end{equation*}
see, e.g., \cite{DP,NW08} for a general reference. We use the following formula for the square worst-case error (see \cite[Proposition~2.11]{DP} or \cite{Hic98a}):
\begin{equation*}
e^2(\mathcal{H}(K);P) = \int_{[0,1]^{2s}} K(\bsx,\bsy) \,\mathrm{d} \bsx \,\mathrm{d} \bsy - \frac{2}{N} \sum_{n=0}^{N-1} \int_{[0,1]^s} K(\bsx,\bsx_n) \,\mathrm{d} \bsx + \frac{1}{N^2}\sum_{n,n'=0}^{N-1} K(\bsx_n,\bsx_{n'}).
\end{equation*}

Integration in the Sobolev space $\mathcal{H}(K^{\sob}_{\alpha,\bsgamma,s})$ has been considered in \cite{D08, DP} and integration in the Korobov space has been studied for instance in \cite{D04,D07,HN2003,kor,kuo,NC06,NC06b}, as well as other papers. In this paper we study numerical integration in the half-period cosine space and in the sum of the Korobov space and the half-period cosine space. For the former space we use tent-transformed lattice rules and for the latter one we use symmetrized lattice rules.

\subsection{Numerical integration in the half-period cosine space}

We now study numerical integration in the half-period cosine space using tent-transformed lattice rules. For a nonnegative real number $x$ we denote the fractional part of $x$ by $\{x\} = x - \lfloor x \rfloor$. For a vector $\bsx$ of nonnegative real numbers, the expression $\{\bsx\}$ denotes the vector of fractional parts. A lattice point set with $N \ge 2$ points and generating vector $\bsg \in \{1,\ldots, N-1\}^s$ is given by
\begin{equation}\label{LR}
  P(\bsg,N)
  :=
  \left\{ \left\{\frac{n \bsg}{N}\right\} : 0 \le n < N \right\}.
\end{equation}
For $x \in [0,1]$ we define the tent-transformation by $\phi(x) = 1-|2x-1|$ and for vectors we apply the function $\phi$ component-wise. The tent-transformed lattice point set is now given by
$$
  P_{\phi}(\bsg,N)
  :=
  \left\{ \phi\left(\left\{\frac{n \bsg}{N}\right\}\right): 0 \le n < N \right\}.
$$
We call a lattice rule which is based on $P_{\phi}(\bsg,N)$ a tent-transformed lattice rule.

The following theorem gives a useful formula for the worst-case
integration error in $\mathcal{H}(K^{\cos}_{\alpha,\bsgamma,s})$ of
tent-transformed lattice rules.
\begin{theorem}\label{col:tentwce}
The squared worst-case error for QMC integration in the half-period
cosine space $\mathcal{H}(K_{\alpha,\bsgamma,s}^{\cos})$ using a
tent-transformed lattice rule is given by
  \begin{equation*}
    e^2(\mathcal{H}(K_{\alpha,\bsgamma,s}^{\cos}); P_{\phi}(\bsg,N))
    =
    \sum_{\bsh \in L^\perp \setminus \{\bszero\}} r_{\alpha,\bsgamma,s}( \bsh)
    ,
  \end{equation*}
 where $L^\perp := \{ \bsh \in \ZZ^s : \bsh \cdot \bsg \equiv 0 \pmod{N} \}$ is the dual lattice.
\end{theorem}
\begin{proof}
Let $f \in \mathcal{H}(K_{\alpha,\bsgamma,s}^{\cos})$ with
$\|f\|_{K_{\alpha,\bsgamma,s}^{\cos}}< \infty$ and with expansion
\begin{equation}\label{eq_f_expansion}
f(\bsx) = \sum_{\bsk \in \mathbb{N}_0^s} \widetilde{f}_{\cos}(\bsk) (\sqrt{2})^{|\bsk|_0} \prod_{j=1}^s \cos(\pi k_i x_i).
\end{equation}
For any $k \in \mathbb{N}_0$ we have
\begin{equation*}
\cos(\pi k \phi(x)) = \cos(2\pi k x) \quad \mbox{for all } x \in
[0,1],
\end{equation*}
and hence
\begin{eqnarray*}
f\left(\phi\left(\left\{ \frac{n \bsg}{N} \right\}\right)\right) & = & \sum_{\bsk \in \NN_0^s} (\sqrt{2})^{|\bsk|_0} \widetilde{f}_{\cos}(\bsk) \prod_{j=1}^s \cos\left(\pi k_j  \phi\left(\left\{ \frac{n g_j}{N} \right\}\right)\right)\\
& = & \sum_{\bsk \in \NN_0^s} (\sqrt{2})^{|\bsk|_0} \widetilde{f}_{\cos}(\bsk) \prod_{j=1}^s \cos\left(2 \pi k_j  \frac{n g_j}{N} \right)\\
& = & \sum_{\bsh \in \ZZ^s} (\sqrt{2})^{-|\bsh|_0} \widetilde{f}_{\cos}(|\bsh|) \, \E^{2 \pi \I n (\bsh \cdot \bsg)/N}.
\end{eqnarray*}
Therefore we obtain
\begin{align*}
    \frac1N \sum_{n=0}^{N-1} f\left(\phi\left(\left\{ \frac{n \bsg}{N} \right\}\right)\right)
    - \int_{[0,1]^s} f(\bsx) \rd \bsx
    &=\sum_{\bszero \ne \bsh \in \ZZ^s} (\sqrt2)^{-|\bsh|_0} \widetilde{f}_{\cos}(|\bsh|)
    \left( \frac1N \sum_{n=0}^{N-1} \E^{2\pi \I \, n (\bsh \cdot \bsg) / N} \right)
    .
\end{align*}
The sum in the braces is a character sum over the group $\ZZ/N\ZZ$ which is one if $\bsh \cdot \bsg$ is a multiple of $N$ and zero otherwise. From this we get
\begin{align}\label{err_fo_tent}
    \frac1N \sum_{n=0}^{N-1} f\left(\phi\left(\left\{ \frac{n \bsg}{N} \right\}\right)\right)
    - \int_{[0,1]^s} f(\bsx) \rd \bsx
    &=
    \sum_{\bsh \in L^\perp \setminus \{\bszero\}} (\sqrt2)^{-|\bsh|_0}\widetilde{f}_{\cos}(|\bsh|)
    .
\end{align}
From this formula and an application of the Cauchy--Schwarz inequality we obtain
\begin{align*}
& \left|\frac1N \sum_{n=0}^{N-1} f\left(\phi\left(\left\{ \frac{n \bsg}{N} \right\}\right)\right)-\int_{[0,1]^s} f(\bsx) \rd \bsx\right|
  \\  &\qquad\qquad=
\left|  \sum_{\bsh \in L^\perp \setminus \{\bszero\} }
r_{\alpha,\bsgamma,s}(\bsh)^{1/2} \, \frac{(\sqrt2)^{-|\bsh|_0}
\widetilde{f}_{\cos}(|\bsh|)}{r_{\alpha,\bsgamma,s}(\bsh)^{1/2}} \right|
    \\
    &\qquad\qquad\le
\left( \sum_{\bsh \in L^\perp \setminus \{\bszero\} }
r_{\alpha,\bsgamma,s}(\bsh) \right)^{1/2} \left( \sum_{\bsh \in
\ZZ^s} \frac{2^{-|\bsh|_0}
|\widetilde{f}_{\cos}(|\bsh|)|^2}{r_{\alpha,\bsgamma,s}(\bsh)}
\right)^{1/2}
    \\
    &\qquad\qquad=
\left( \sum_{\bsh \in L^\perp \setminus \{\bszero\}}
r_{\alpha,\bsgamma,s}(\bsh) \right)^{1/2}
    \| f \|_{K_{\alpha,\bsgamma,s}^{\cos}}
    .
  \end{align*}
Here we obtain equality by~\eqref{err_fo_tent} for the function with cosine series coefficients given by
$$
  \widetilde{f}_{\cos}(\bsk)
  =
  \begin{cases}
     (\sqrt{2})^{|\bsk|_0} r_{\alpha,\bsgamma,s}(\bsk)
       & \mbox{for } \bsk \in \{ |\bsh| : \bsh \in L^\perp \setminus \{\bszero\} \} \\
     0 & \mbox{otherwise}.
   \end{cases}
$$
Hence the result follows by the definition of the worst-case error.
\end{proof}

The above result shows that the square worst-case error of tent-transformed lattice rules for
numerical integration in the cosine space coincides with the square
worst-case error for numerical integration in a Korobov space using the same lattice rules but without applying the tent-transformation, that is,
\begin{eqnarray*}
e^2(\mathcal{H}(K^{\kor}_{\alpha,\bsgamma,s}); P(\bsg,N)) & = & -1 + \sum_{\bsh \in \mathbb{Z}^s} r_{\alpha,\bsgamma,s}(\bsh) \left|\frac{1}{N} \sum_{n=0}^{N-1} \mathrm{e}^{2\pi \mathrm{i} n \bsh \cdot \bsg/N} \right|^2 \\ & = & \sum_{\bsh \in L^\perp\setminus \{\bszero\}} r_{\alpha,\bsgamma,s}(\bsh) \\ & = & e^2(\mathcal{H}(K_{\alpha,\bsgamma,s}^{\cos}); P_{\phi}(\bsg,N)).
\end{eqnarray*}
Thus all the results which hold for the worst-case error in the Korobov space using lattice rules, also hold for the worst-case error in the half-period cosine space using tent-transformed lattice rules. This applies for instance to the component-by-component construction \cite{D04,kuo,SR,SKJ} and fast component-by-component construction \cite{NC06,NC06b}, general weights \cite{DSWW,KSS} and extensible lattice rules \cite{CKN2006,DPW,HKKN2011,HN2003}.

\begin{corollary}
Using the fast component-by-component algorithm one can obtain a generating vector $\bsg \in \{1,\ldots, N-1\}^s$ such that
$$e(\mathcal{H}(K_{\alpha,\bsgamma,s}^{\cos});P_{\phi}(\bsg,N)) \le C_{\alpha,\bsgamma,s,\tau} (N-1)^{-\tau/2},$$
for all $1 \le \tau < 2\alpha$, where the constant $C_{\alpha,\bsgamma,s,\tau} > 0$ is given by
\begin{equation*}
C_{\alpha,\bsgamma,s,\tau} = \left(\sum_{\emptyset \neq \uu \subseteq [s]} \gamma_{\uu}^{1/\tau} (2 \zeta(2\alpha/\tau))^{|\uu|} \right)^{\tau/2}=\left(-1+\prod_{j=1}^s(1+2 \zeta(2\alpha/\tau) \gamma_j^{1/\tau})\right)^{\tau/2}.
\end{equation*}
\end{corollary}
Note that for certain choices of weights $\bsgamma$, the upper bound can be made independent of the dimension $s$ and we then obtain (strong) tractability results. See \cite{D04,DSWW,kuo,NW08} for a discussion of tractability results which apply in this context.

Since the cosine series space $\mathcal{H}(K^{\cos}_{1,\gamma,s})$ and
the unanchored Sobolev space $\mathcal{H}(K^{\sob}_{1,\bsgamma,s})$ coincide, we also get the following
result.
\begin{corollary}
Using the fast component-by-component algorithm one can obtain a generating vector $\bsg \in \{1,\ldots, N-1\}^s$ such that
\begin{equation*}
e(\mathcal{H}(K^{\sob}_{1,\bsgamma,s}), P_{\phi}(\bsg,N)) \le C_{1,\bsgamma,s,\tau} (N-1)^{-\tau/2},
\end{equation*}
for all $1\le \tau < 2$, where
\begin{equation*}
C_{1,\bsgamma,s,\tau} = \left(\sum_{\emptyset \neq \uu \subseteq [s]} \gamma_{\uu}^{1/\tau} (2 \zeta(2/\tau))^{|\uu|} \right)^{\tau/2}=\left(-1+\prod_{j=1}^s(1+2 \zeta(2/\tau) \gamma_j^{1/\tau})\right)^{\tau/2}.
\end{equation*}
\end{corollary}

\begin{remark}
An analoguous result can also be obtained for the space considered in \cite{WW09}, since the eigenfunctions are the same as for $H(K^{\sob}_{1,\gamma})$.
\end{remark}

This result is a deterministic version of the main results in \cite{D04,kuo}, where a random shift was required to achieve this bound. The tractability results of \cite{D04,DSWW,kuo} also apply.

\subsection{Numerical integration in the Korobov plus half-period cosine space}
\label{sec:symm}

We now study numerical integration in the space $\mathcal{H}(K^{\kor+\cos}_{\alpha,\bsgamma,s})$ using symmetrized lattice rules.
Let $\bsx = (x_1,\ldots, x_s)$ and let $\uu \subseteq [s]$. Then let $\sym_{\uu}(\bsx)$ denote the vector whose $j$th coordinate is $x_j$ if $j \in \uu$ and $1-x_j$ otherwise, i.e., $\sym_{\uu}(\bsx) = (y_1,\ldots, y_s)$ with
\begin{align*}
y_j  &=
\begin{cases}
    1-x_j & \text{if } j \in \uu,\\
    x_j   & \text{if } j \not\in \uu.
\end{cases}
\end{align*}
For a lattice point set as in \eqref{LR} let
$$
  P_{\sym}(\bsg,N)
  :=
  \left\{ \sym_{\uu}\left(\left\{\frac{n \bsg}{N}\right\}\right): 0 \le n < N, \uu \subseteq [s] \right\}.$$
We call a lattice rule which is based on $P_{\sym}(\bsg,N)$ a symmetrized lattice rule.
Note that $P_{\sym}(\bsg,N)$ consists of $O(2^{s-1}N)$ elements as we show next.

\begin{lemma}
  The number of nodes in the symmetrized lattice rule $P_{\sym}(\bsg,N)$ is given by $2^{s-1} (N+1)$ if $2 \nmid N$ and $2^{s-1} N + 1$ if $2 \mid N$.
\end{lemma}
\begin{proof}
  The argument comes from \cite{Zar72a}.
  We have the following symmetry
  $$
    k g_j \equiv N - (N-k) g_j \pmod{N}, \quad \text{for all } 0 \le k < N \text{ and } j=1,\ldots,s,
  $$
  which corresponds exactly to $x_{k,j} = 1 - x_{N-k,j}$.
  This means we only have to evaluate and symmetrize the points for $0 \le k \le N/2$.
  \begin{enumerate}
  \item For $0 < k < N/2$ we have $2^s (N-1)/2$ points if $2 \nmid N$ and $2^s (N/2-1)$ if $2 \mid N$. 
  \item For $k = 0$ symmetrization returns all $2^s$ corner points.
  \item If $2 \mid N$ then for $k = N/2$ we have $\bsx_{N/2} = (\tfrac12, \ldots, \tfrac12)$ and thus no symmetrization is needed.
  \end{enumerate}
  Counting the number of function evaluations now gives the result from above.
\end{proof}
The analysis used in the previous proof can also be used in an implementation where further stream lining can be done by noticing that $x \mapsto 1 - x$ is its own inverse and thus the $2^s$ symmetric points for $0 < k < N/2$ can be constructed easily in gray code ordering.
Nevertheless, we increase the number of points by a factor of $2^{s-1}$ and so this is only feasible for moderate dimensions.
For the derivations below we symmetrize all $N$ points for ease of notation.

The following theorem gives a useful formula for the
worst-case integration error in
$\mathcal{H}(K_{\alpha,\bsgamma,s}^{\kor+\cos})$ of symmetrized lattice
rules.

\begin{theorem}\label{symLR}
The squared worst-case error for QMC integration in the sum of the half-period
cosine space and the Korobov space $\mathcal{H}(K_{\alpha,\bsgamma,s}^{\kor+\cos})$ using a
symmetrized lattice rule is given by
$$e^2(\mathcal{H}(K_{\alpha,\bsgamma,s}^{\kor + \cos});
P_{\sym}(\bsg,N))= \sum_{\bsh \in L^\perp \setminus\{\bszero\}}
r_{\alpha,\bsgamma,s}(\bsh),$$ where $L^\perp := \{ \bsh \in
\ZZ^s: \bsh \cdot \bsg \equiv 0 \pmod{N} \}$ is the dual lattice.
\end{theorem}

\begin{proof}
Let $f \in \mathcal{H}(K_{\alpha,\bsgamma,s}^{\kor + \cos})$ with
$\|f\|_{K_{\alpha,\bsgamma,s}^{\kor+\cos}} < \infty$. Let $\vv \subseteq [s]$, then by \cite[Part~I, Section~6]{A50} there are functions $f_{\vv} \in \mathcal{H}(K^{\kor+\cos}_{\alpha,\bsgamma,s,\vv}) = \left( \bigotimes_{j\in \vv} \mathcal{H}(\tfrac12 K^{\kor}_{\alpha,\gamma_j}) \right) \otimes \left( \bigotimes_{j \in [s]\setminus \vv} \mathcal{H}(\tfrac12 K^{\cos}_{\alpha,\gamma_j}) \right)$ given by
\begin{equation*}
f_{\vv}(\bsx) = \sum_{\bsh_{\vv} \in \mathbb{Z}^{|\vv|}} \sum_{\bsk_{[s] \setminus \vv} \in \mathbb{N}_0^{s-|\vv|}} \widetilde{f}_{\vv,\kor+\cos}(\bsh_{\vv},\bsk_{[s]\setminus \vv}) \, (\sqrt2)^{|\bsk_{[s]\setminus \vv}|_0} \prod_{j\in \vv} \mathrm{e}^{2\pi \mathrm{i} h_j x_j} \prod_{j \in [s]\setminus \vv} \cos(\pi k_j x_j)
\end{equation*}
such that
\begin{equation*}
\|f\|^2_{K^{\kor+\cos}_{\alpha,\bsgamma,s}} = \sum_{\vv \subseteq [s]} \|f_{\vv}\|^2_{K^{\kor+\cos}_{\alpha,\bsgamma,s,\vv}}.
\end{equation*}

Note that for $k \in \ZZ$ and $x \in \RR$,
\begin{align*}
\cos(\pi k x) + \cos(\pi k (1-x))
    &=
    \begin{cases}
      2 \cos(\pi k x) =\mathrm{e}^{\pi \mathrm{i} k x} + \mathrm{e}^{-\pi \mathrm{i} kx}     & \text{if $k$ is even}, \\
      0                 & \text{if $k$ is odd}.
    \end{cases}
  \end{align*}
For given $\vv \subseteq [s]$ we therefore have
\begin{eqnarray*}
\lefteqn{\frac{1}{2^sN} \sum_{n=0}^{N-1} \sum_{\uu \subseteq [s]}f_{\vv}\left( \sym_{\uu}\left( \left\{ \frac{n \bsg}{N} \right\} \right) \right)}\\ &= & \frac{1}{2^sN} \sum_{n=0}^{N-1} \sum_{\bsh_{\vv} \in \mathbb{Z}^{|\vv|}}\sum_{\bsk_{[s]\setminus \vv} \in \NN_0^{s-|\vv|}}
    (\sqrt2)^{|\bsk_{[s]\setminus \vv}|_0} \widetilde{f}_{\vv,\kor+\cos}(\bsh_{\vv}, \bsk_{[s]\setminus \vv})
    \\ & & \times \prod_{j\in \vv} \left( \mathrm{e}^{2\pi \mathrm{i} h_j n g_j /N} + \mathrm{e}^{-2\pi \mathrm{i} h_j n g_j/N} \right) \prod_{j \in [s]\setminus \vv}
    \left( \cos(\pi k_j \{n g_j / N\})  +  \cos(\pi k_j (1-\{n g_j / N\})) \right)  \\ & = &
    \frac{1}{2^s N} \sum_{n=0}^{N-1} \sum_{\bsh_{\vv} \in \mathbb{Z}^{|\vv|}}
    \sum_{\bsk_{[s]\setminus \vv} \in \NN_0^{s-|\vv|}}
    (\sqrt2)^{|\bsk_{[s]\setminus \vv}|_0} \widetilde{f}_{\vv,\kor+\cos}(\bsh_{\vv}, 2 \bsk_{[s]\setminus \vv})
    \\ & & \times  \prod_{j\in \vv} \left( \mathrm{e}^{2\pi \mathrm{i} h_j n g_j /N} + \mathrm{e}^{-2\pi \mathrm{i} h_j n g_j/N} \right) \prod_{j \in [s]\setminus \vv}  \left( \mathrm{e}^{2\pi \mathrm{i} k_j n g_j /N} + \mathrm{e}^{-2\pi \mathrm{i} k_j n g_j/N} \right)
     \\& = & \frac{1}{2^s N} \sum_{n=0}^{N-1} \sum_{\bsh \in \mathbb{Z}^{s}}
    2^{|\vv|} (\sqrt2)^{|\bsh_{[s]\setminus \vv}|_0} 2^{s-|\vv|-|\bsh_{[s]\setminus \vv}|_0} \widetilde{f}_{\vv,\kor+\cos}(\bsh_{\vv},2 |\bsh_{[s]\setminus \vv}|) \mathrm{e}^{2\pi \mathrm{i} n \bsh \cdot \bsg /N} \\& = &  \sum_{\bsh \in \mathbb{Z}^{s}}
    (\sqrt2)^{-|\bsh_{[s]\setminus \vv}|_0}  \widetilde{f}_{\vv,\kor+\cos}(\bsh_{\vv},2 |\bsh_{[s]\setminus \vv}|) \left(\frac{1}{N} \sum_{n=0}^{N-1} \mathrm{e}^{2\pi \mathrm{i} n \bsh \cdot \bsg /N} \right).
  \end{eqnarray*}
The sum in the braces is a character sum over the group $\ZZ/N\ZZ$ which is one if $\bsh \cdot \bsg$ is a multiple of $N$ and zero otherwise. From this we get
\begin{multline}\label{err_f_sym}
\frac{1}{2^sN} \sum_{n=0}^{N-1} \sum_{\uu \subseteq [s]}
f_{\vv}\left( \sym_{\uu}\left( \left\{ \frac{n \bsg}{N} \right\} \right) \right) - \int_{[0,1]^s} f_{\vv}(\bsx) \,\mathrm{d} \bsx
\\=  \sum_{\bsh \in L^{\perp} \setminus \{\bszero\} } (\sqrt2)^{-|\bsh_{[s]\setminus \vv}|_0} \widetilde{f}_{\vv,\kor+\cos}(\bsh_{\vv},2 |\bsh_{[s]\setminus \vv}|).
\end{multline}
Thus we find
\begin{eqnarray*}
   \lefteqn{\left| \frac{1}{2^sN} \sum_{n=0}^{N-1} \sum_{\uu \subseteq [s]} f\left( \sym_{\uu}\left( \left\{ \frac{n \bsg}{N} \right\} \right) \right) - \int_{[0,1]^s} f(\bsx) \rd \bsx \right|}
    \\
    & \le &  \sum_{\vv \subseteq [s]} \sum_{\bsh \in L^{\perp} \setminus \{\bszero\} }  (\sqrt2)^{-|\bsh_{[s]\setminus \vv}|_0} | \widetilde{f}_{\vv,\kor+\cos}(\bsh_{\vv}, 2|\bsh_{[s]\setminus \vv}|) |
    \\
    &= &
    \sum_{\vv \subseteq [s]} \sum_{\bsh \in L^{\perp} \setminus \{\bszero\} } r_{\alpha,\bsgamma,s}(\bsh)^{1/2}   \frac{(\sqrt2)^{-|\bsh_{[s] \setminus \vv}|_0} | \widetilde{f}_{\vv,\kor+\cos}(\bsh_{\vv}, 2|\bsh_{[s] \setminus \vv}|) |}{r_{\alpha,\bsgamma,s}(\bsh)^{1/2}}
    \\
    &\le &
    \left(\sum_{\vv \subseteq [s]} \sum_{\bsh \in L^{\perp} \setminus \{\bszero\} } r_{\alpha,\bsgamma,s}(\bsh) \right)^{1/2}
    \left(\sum_{\vv \subseteq [s]} \sum_{\bsh \in L^{\perp} \setminus \{\bszero\} } \frac{2^{-|\bsh_{[s]\setminus \vv}|_0} | \widetilde{f}_{\vv,\kor+\cos}(\bsh_{\vv}, 2|\bsh_{[s]\setminus \vv}|) |^2}{r_{\alpha,\bsgamma,s}(\bsh)}\right)^{1/2}
    \\
    &\le &
    \left(2^s \sum_{\bsh \in L^{\perp} \setminus \{\bszero\} } r_{\alpha,\bsgamma,s}(\bsh) \right)^{1/2}
    \left(\frac{1}{2^s} \sum_{\vv \subseteq [s]} 2^s \sum_{\bsh_\vv \in \mathbb{Z}^{|\vv|} } \sum_{\bsk_{[s]\setminus \vv} \in \mathbb{N}_0^{s-|\vv|}} \frac{| \widetilde{f}_{\vv,\kor+\cos}(\bsh_{\vv}, \bsk_{[s]\setminus \vv}) |^2}{r_{\alpha,\bsgamma,s}(\bsh_\vv, \bsk_{[s]\setminus \vv})}\right)^{1/2}
    \\
    &\le &
    \left(2^s \sum_{\bsh \in L^{\perp} \setminus \{\bszero\} } r_{\alpha,\bsgamma,s}(\bsh) \right)^{1/2}
    \left(\frac{1}{2^s} \sum_{\vv \subseteq [s]} \| f \|^2_{K_{\alpha,\bsgamma,s,\vv}^{\kor+\cos}} \right)^{1/2}
    \\
    & \le & 
    \left(\sum_{\bsh \in L^{\perp} \setminus \{\bszero\} } r_{\alpha,\bsgamma,s}(\bsh) \right)^{1/2} \|f\|_{K^{\kor+\cos}_{\alpha,\bsgamma,s}}
    .
\end{eqnarray*}
Thus the result follows.
\end{proof}

Again we can relate the worst-case error $e^2(\mathcal{H}(K^{\kor+\cos}_{\alpha,\bsgamma,s});
P_{\sym}(\bsg,N))$ to the worst-case error in a Korobov space. We have
\begin{equation*}
e(\mathcal{H}(K^{\kor+\cos}_{\alpha,\bsgamma,s}); P_{\sym}(\bsg,N)) = e(\mathcal{H}(K^{\kor}_{\alpha,\bsgamma,s}); P(\bsg,N)).
\end{equation*}
Thus the results for integration in the sum of the half-period cosine space and the Korobov space using symmetrized lattice rules are the same as in a Korobov space using lattice rules. In particular, the component-by-component algorithm can be used \cite{D04,kuo,SR,SKJ} and also its fast version \cite{NC06,NC06b}, general weights \cite{DSWW,KSS} and extensible lattice rules \cite{CKN2006,DPW,HKKN2011,HN2003}.

\begin{corollary}
Using the fast component-by-component algorithm one can obtain a generating vector $\bsg \in \{1,\ldots, N-1\}^s$ such that
$$e(\mathcal{H}(K_{\alpha,\bsgamma,s}^{\kor+\cos});P_{\sym}(\bsg,N)) \le C_{\alpha,\bsgamma,s,\tau} (N-1)^{-\tau/2},$$
for all $1 \le \tau < 2\alpha$, where the constant $C_{\alpha,\bsgamma,s,\tau} > 0$ is given by
\begin{equation*}
C_{\alpha,\bsgamma,s,\tau} = \left(\sum_{\emptyset \neq \uu \subseteq [s]} \gamma_{\uu}^{1/\tau} (2 \zeta(2\alpha/\tau))^{|\uu|} \right)^{\tau/2}=\left(-1+\prod_{j=1}^s(1+2 \zeta(2\alpha/\tau) \gamma_j^{1/\tau})\right)^{\tau/2}.
\end{equation*}
\end{corollary}
Due to the fact that the number of points of $P_{\sym}(\bsg,N)$ is $M = O(2^{s-1} N)$ we do not get tractability results. Notice that in terms of the number of points one gets $(N-1)^{-\tau/2}  \approx 2^{(s-1)\tau/2} M^{-\tau/2}$, which also implies a strong dependence on the dimension.

A consequence of the symmetrization procedure is that
all functions of the form
\begin{align*}
  \sum_{\substack{ k_1,\ldots,k_s \in \NN \\
                   k_1,\ldots,k_s \text{ odd} }}
  b_{k_1,\ldots, k_s} \prod_{j=1}^s \cos(\pi k_j x_j)
   &&\text{for all }
   b_{k_1,\ldots, k_s} \in \RR
\end{align*}
are integrated exactly.
Likewise, all polynomials of the form
\begin{align*}
  \sum_{\substack{k_1,\ldots, k_s \in \NN \\
                  k_1,\ldots, k_s \text{ odd}}}
   a_{k_1,\ldots, k_s} \prod_{j=1}^s (x_j-1/2)^{k_j}
   &&\text{for all }
   a_{k_1,\ldots, k_s} \in \RR
\end{align*}
are integrated exactly.
This is because all the odd frequencies in a cosine series are integrated exactly.
Specifically for the half-period cosine space we can state the following result.
\begin{corollary}
The squared worst-case error for QMC integration in the half-period
cosine space $\mathcal{H}(K_{\alpha,\bsgamma,s}^{\cos})$ using a
symmetrized lattice rule is given by
$$e^2(\mathcal{H}(K_{\alpha,\bsgamma,s}^{\cos});
P_{\sym}(\bsg,N))= \sum_{\bsh \in L^\perp \setminus\{\bszero\}}
r_{\alpha,\bsgamma,s}(2 \bsh),$$ where $L^\perp := \{ \bsh \in
\ZZ^s: \bsh \cdot \bsg \equiv 0 \pmod{N} \}$ is the dual lattice.
\end{corollary}
\begin{proof}
  Similar to the proof of Theorem~\ref{symLR} with $\vv = \emptyset$.
\end{proof}

\subsection{A lower bound on the worst-case error}\label{sec_lower_bound}

We prove the following lower bound for integration in the half-period cosine space. Let $P = \{\boldsymbol{x}_0,\ldots, \boldsymbol{x}_{N-1}\} \subseteq [0,1]^s$ be an $N$ element point set and let $\boldsymbol{w} = (w_0,\ldots, w_{N-1})$ be an arbitrary real tuple. Let
\begin{equation*}
e(\mathcal{H}(K^{\cos}_{\alpha,\bsgamma,s}); P; \boldsymbol{w}) = \sup_{\satop{f \in \mathcal{H}(K^{\cos}_{\alpha,\bsgamma,s})}{\|f\|_{K^{\cos}_{\alpha,\bsgamma,s}} \le 1}} \left|\int_{[0,1]^s} f(\boldsymbol{x}) \,\mathrm{d} \boldsymbol{x} - \sum_{n=0}^{N-1} w_n f(\boldsymbol{x}_n) \right|.
\end{equation*}

\begin{theorem} \label{thm_lower_bound_cosine}
For $P$ an arbitrary $N$-element point set in $[0,1]^s$ and $\boldsymbol{w} = (w_0,\dots, w_{N-1}) \in \mathbb{R}^N$ we have
\begin{equation*}
e(\mathcal{H}(K^{\cos}_{\alpha,\bsgamma,s}); P; \boldsymbol{w}) \ge C(\alpha,\bsgamma,s) \frac{(\log N)^{(s-1)/2}}{N^{\alpha}}
\end{equation*}
where $C(\alpha,\bsgamma,s) > 0$ depends on $\alpha, \bsgamma$ and $s$, but not on $N$ and $\boldsymbol{w}$.
\end{theorem}

\begin{proof}
We follow the proof of Temlyakov~\cite[Lemma~3.1]{tem}.
Let $\beta := \sum_{n=0}^{N-1} w_n$.
If $\beta = 0$ then $e(\mathcal{H}(K^{\cos}_{\alpha,\bsgamma,s}); P; \boldsymbol{w}) \ge 1$, since for $f = 1$ the integration error is $1$.
In this case the result holds trivially. Thus we can assume now that $\beta \neq 0$.

We have
\begin{eqnarray*}
  \lefteqn{ e^2(\mathcal{H}(K^{\cos}_{\alpha,\bsgamma,s}); P; \boldsymbol{w}) }
  \\
  & = &
  \int_{[0,1]^{2s}} K(\boldsymbol{x}, \boldsymbol{y}) \rd{\bsx} \rd{\bsy}
  - 2 \sum_{n=0}^{N-1} w_n \int_{[0,1]^s} K(\boldsymbol{x},\boldsymbol{x}_n) \rd{\bsx}
  + \sum_{n,n'=0}^{N-1} w_n w_{n'} K(\bsx_n,\bsx_{n'})
  \\
  & = &
  (1-\beta)^2
  + \sum_{\bsk \in \mathbb{N}_0^s \setminus \{\bszero\}} r_{\alpha,\bsgamma,s}(\bsk) \, 2^{|\bsk|_0} \left( \sum_{n=0}^{N-1} w_n \prod_{j=1}^s \cos(\pi k_j x_{j,n}) \right)^2
\end{eqnarray*}
where $\bsx_n = (x_{1,n},\ldots, x_{s,n})$.

For $\boldsymbol{m} = (m_1,\ldots, m_s) \in \mathbb{N}_0^s$ and $|\bsm| := m_1 + \cdots + m_s$
we will now construct a function $G(\bsy) := \sum_{|\bsm| = t} F_{\bsm}(\bsy)$, parametrized by the points $\bsx_n$ and weights $w_n$ of the arbitrary cubature rule, to obtain a lower bound on the worst-case error.
For this we will pick its cosine coefficients to be bounded above by $r_{\alpha,\bsgamma,s}(\bsk)$.
Let the integer $t$ be chosen such that
$$
  2N \le 2^t < 4N.
$$

Let $a := \lceil \alpha \rceil + 1$ and let $f:\mathbb{R} \to \mathbb{R}$ be the $a$-times differentiable function
\begin{equation}\label{def_f_lower_bound}
  f(x)
  :=
  \begin{cases}
    x^{a+1} (1-x)^{a+1} & \mbox{for } 0 < x < 1,  \\
    0 & \mbox{otherwise}.
  \end{cases}
\end{equation}
Note that $f(x) > 0$ for $0 < x < 1$ and $\supp(f^{(\tau)}) = (0, 1)$ for all $0 \le \tau \le a$.

For $m \in \NN_0$ let $f_m(x) := f(2^{m+2}x)$ and for $\boldsymbol{m} = (m_1,\ldots, m_s) \in \mathbb{N}_0^s$ and $\bsx = (x_1,\ldots, x_s) \in \RR^s$ let
\begin{equation*}
  f_{\bsm}(\bsx) := \prod_{j=1}^s f_{m_j}(x_j).
\end{equation*}
Then $\supp(f_{\bsm}) = \prod_{j=1}^s (0, 2^{-m_j-2})$.
We obtain
\begin{equation*}
  \widetilde{f}_{\bsm, \cos}(\bszero) = \prod_{j=1}^s \int_0^1 f(2^{m_j+2} x) \,\mathrm{d} x = \prod_{j=1}^s \frac{1}{2^{m_j+2}} \int_0^1 f(y) \,\mathrm{d} y = \frac{1}{2^{|\bsm| + 2s}} \left(I(f)\right)^s,
\end{equation*}
where $I(f) := \int_0^1 f(y) \,\mathrm{d} y$. For $f$ given by \eqref{def_f_lower_bound} we obtain $$I(f) = B(a+2,a+2) = \frac{((a+1)!)^2}{(2a+3)!},$$ where $B$ denotes the beta function.

For $k \not=0$ we have
\begin{eqnarray*}
  \widetilde{f}_{m,\cos}(k)
  & = &
  \int_{0}^1 f(2^{m+2} x) \sqrt{2} \cos(\pi k x) \,\mathrm{d} x
  \\
  & = &
  \frac{1}{\sqrt{2}} \int_{0}^{2^{-m-2}} f(2^{m+2} x) \left( \mathrm{e}^{\pi \mathrm{i} k x} + \mathrm{e}^{-\pi \mathrm{i} k x}\right) \, \mathrm{d} x
  \\
  & = &
  \frac{1}{2^{m+2} \sqrt{2}} \int_{0}^1 f(y) \left( \mathrm{e}^{2\pi \mathrm{i} k 2^{-m-3} y} + \mathrm{e}^{-2\pi \mathrm{i} k 2^{-m-3} y} \right) \,\mathrm{d} y
  \\
  & = &
  \frac{\widehat{f}(k 2^{-m-3}) + \widehat{f}(-k 2^{-m-3})}{2^{m+2} \sqrt{2}},
\end{eqnarray*}
where
$$
  \widehat{f}(h)
  =
  \int_0^1 f(x) \, \mathrm{e}^{-2\pi \mathrm{i} h x} \,\mathrm{d} x
$$
denotes the Fourier transform of $f$. Since, by definition, $f^{(\tau)}(0) = f^{(\tau)}(1) = 0$ for all $0 \le \tau \le a$, and $f$ is $a$-times differentiable, repeated integration by parts shows that for any $m \in \mathbb{N}_0$ we have
\begin{equation*}
|\widehat{f}(k 2^{-m-3})|  \le C_a \min(1, (k 2^{-m-3})^{-a}),
\end{equation*}
where the constant $C_a > 0$ depends only on $a$ (and $f$). Thus we have
\begin{eqnarray*}
|\widetilde{f}_{m,\cos}(k)| & \le & C_a \, 2^{-m-3/2} \min(1, (k 2^{-m-3})^{-a})\\
& \le & C'_a \, 2^{-m} \min(1,2^{a m} r_{a/2,1}(k)).
\end{eqnarray*}
This bound even holds for $\widetilde{f}_{m,\cos}(0)$ if $C'_a$ is large enough.
For the multivariate case we have the bound
\begin{eqnarray*}
  |\widetilde{f}_{\bsm,\cos}(\bsk)|
  & \le &
  C(a,s) \prod_{j=1}^s 2^{-m_j} \min(1, 2^{a m_j} r_{a/2,1}(k_j))
  \\
  & = &
  C(a,s) \, 2^{(\alpha - 1) |\bsm|} \prod_{j=1}^s 2^{-\alpha m_j} \min(1, 2^{a m_j} r_{a/2,1}(k_j))
  .
\end{eqnarray*}
By summing $|\widetilde{f}_{\bsm,\cos}(\bsk)|^2$ over all choices of $\bsm$ where $|\bsm| = t$ we obtain
\begin{eqnarray*}
  \sum_{\bsm \in \NN_0^s \atop |\bsm| = t} |\widetilde{f}_{\bsm,\cos}(\bsk)|^2
  & \le &
  2^{2(\alpha-1) t} C^2(a,s)  \sum_{\bsm\in \NN_0^s \atop |\bsm| = t}  \prod_{j=1}^s 2^{-2\alpha m_j} \min(1, 2^{2a m_j} r_{a,1}(k_j))
  \\
  & \le &
  2^{2(\alpha-1) t} C^2(a,s) \prod_{j=1}^s \sum_{m=0}^\infty 2^{-2\alpha m} \min(1,2^{2a m} r_{a,1}(k_j))
  .
\end{eqnarray*}
The last sum can now be estimated by
\begin{eqnarray*}
  \lefteqn{
    \sum_{m=0}^\infty 2^{-2\alpha m} \min(1,2^{2a m} r_{a,1}(k_j))
  }
  \\
  & = &
  \sum_{0 \le m \le (\log_2 r^{-1}_{a,1}(k_j))/2a} 2^{2(a-\alpha) m} r_{a,1}(k_j)+ \sum_{m > (\log_2 r^{-1}_{a,1}(k_j))/2a} 2^{-2\alpha m}
  \\
  & \le &
  \frac{r^{-1}_{a-\alpha,1}(k_j) 2^{2(a-\alpha)}-1}{2^{2(a-\alpha)}-1} r_{a,1}(k_j)  + \frac{r_{\alpha,1}(k_j) 2^{2\alpha}}{2^{2\alpha}-1}
  \\
  & \le &
  r_{\alpha,1}(k_j) \left(1 + \frac{2^{2\alpha}}{2^{2\alpha}-1}\right)
  \\
  &\le &
  3 \, r_{\alpha,1}(k_j).
\end{eqnarray*}
Thus, since $2N \le 2^t < 4N$, we have
\begin{equation*}
  r_{\alpha,\bsgamma,s}(\bsk)
  \ge
  C_0(a,\bsgamma,s) 2^{-2(\alpha - 1) t} \sum_{\bsm \in \NN_0^s \atop |\bsm| = t} |\widetilde{f}_{\bsm,\cos}(\bsk)|^2
  \ge
  C_1(a,\bsgamma,s) \frac{2^{2t}}{N^{2\alpha}} \sum_{\bsm \in \NN_0^s \atop |\bsm| = t} |\widetilde{f}_{\bsm,\cos}(\bsk)|^2
  .
\end{equation*}

Now for $\bsx =(x_1,\ldots,x_s)$ and  $\bsy=(y_1,\ldots,y_s)$ and for $\uu \subseteq [s]$ define
$$
  \bsx \pmu \bsy
  =
  (z_1,\ldots,z_s),
$$
where $z_j=x_j+y_j$ if $j \in \uu$ and $z_j=x_j-y_j$ if $j \not\in \uu$.
Define the functions
\begin{align*}
  F_{\bsm,\uu}(\bsy)
  &:=
  \sum_{n=0}^{N-1} w_n f_{\bsm}(\bsx_n \pmu \bsy)
  ,
  &
  F_{\bsm}(\bsy)
  &:=
  \frac{1}{2^s} \sum_{\uu \subseteq [s]} F_{\bsm,\uu}(\bsy)
  ,
\end{align*}
and the sets
\begin{align*}
  B_{\bsm,\uu}
  &:=
  \left\{ \bsy \in [0,1]^s: F_{\bsm,\uu}(\bsy) = 0 \right\}
  ,
  \\
  B_{\bsm}
  &:=
  \left\{ \bsy \in [0,1]^s: F_{\bsm,\uu}(\bsy) = 0  \mbox{ for all } \uu \subseteq [s] \right\}
  =
  \bigcap_{\uu \subseteq [s]} B_{\bsm,\uu}
  .
\end{align*}
Denote with $B_{\bsm,\uu}^{c}$ the complement with respect to $[0,1]^s$.
Then for $\lambda_s$ the $s$-dimensional Lebesgue measure we have $\lambda_s(\supp(F_{\bsm,\uu})) = \lambda_s(B_{\bsm,\uu}^c)$.
Since $\supp(f_{\bsm}(\bsx_n \pmu \bsy))$ as a function of $\bsy$ is contained in the interval $\prod_{j \in \uu}(-x_{j,n},-x_{j,n}+2^{-m_j-2}) \prod_{j \in [s] \setminus \uu} (x_{j,n} - 2^{-m_j-2}, x_{j,n})$ we have
$$
  \supp(F_{\bsm,\uu})
  \subseteq
  \bigcup_{n=0}^{N-1}
    \prod_{j \in \uu}(-x_{j,n},-x_{j,n}+2^{-m_j-2})
    \prod_{j \in [s]\setminus\uu} (x_{j,n} - 2^{-m_j-2}, x_{j,n})
  .
$$
Thus $\lambda_s(\supp(F_{\bsm,\uu})) = \lambda_s(B_{\bsm,\uu}^c) \le N 2^{-|\bsm|-2s}$.
Now, for all $\bsm$ satisfying $|\bsm|=t$ we obtain
$$
  \lambda_s(B_{\bsm})
  =
  1 - \lambda_s(B_{\bsm}^c)
  =
  1 - \lambda_s\left( \bigcup_{\uu \subseteq [s]} B_{\bsm,\uu}^c \right)
  \ge
  1 - \sum_{\uu \subseteq [s]} N 2^{-|\bsm|-2s}
  =
  1 - \frac{N}{2^{|\bsm|+s}}
  >
  1/2,
$$
since $2 N \le 2^t < 4N$.

We can expand $F_{\bsm}(\bsy) - \int_{[0,1]^s} F_{\bsm}(\bsy) \rd{\bsy}$ in terms of the coefficients $\widetilde{f}_{\bsm,\cos}(\bsk)$:
\begin{eqnarray*}
  \lefteqn{\frac{1}{2^s} \sum_{\uu \subseteq [s]}\sum_{n=0}^{N-1} w_n f_{\bsm}(\bsx_n \pmu \bsy) - \widetilde{f}_{\bsm,\cos}(\bszero) \beta}
  \\
  & = &
  \sum_{n=0}^{N-1} w_n \sum_{\bsk \in \NN_0^s \setminus\{\bszero\}} \widetilde{f}_{\bsm,\cos}(\bsk) \frac{(\sqrt{2})^{|\bsk|_0}}{2^s} \sum_{\uu \subseteq [s]} \prod_{j\in \uu} \cos(\pi k_j (x_{j,n} + y_j)) \prod_{j\in [s]\setminus \uu} \cos(\pi k_j (x_{j,n} - y_j))
  \\
  & = &
  \sum_{n=0}^{N-1} w_n \sum_{\bsk \in \NN_0^s \setminus\{\bszero\}} \widetilde{f}_{\bsm,\cos}(\bsk) (\sqrt{2})^{|\bsk|_0} \prod_{j=1}^s \frac{\cos(\pi k_j(x_{j,n}-y_j))+\cos(\pi k_j (x_{j,n}+y_j))}{2}
  \\
  & = &
  \sum_{n=0}^{N-1} w_n \sum_{\bsk \in \NN_0^s \setminus\{\bszero\}} \widetilde{f}_{\bsm,\cos}(\bsk) (\sqrt{2})^{|\bsk|_0} \prod_{j=1}^s (\cos(\pi k_j x_{j,n}) \cos(\pi k_j y_j))
  \\
  & = & \sum_{\bsk \in \NN_0^s \setminus\{\bszero\}} \widetilde{f}_{\bsm,\cos}(\bsk) (\sqrt{2})^{|\bsk|_0} \left(\sum_{n=0}^{N-1} w_n \prod_{j=1}^s \cos(\pi k_j x_{j,n})\right) \prod_{j=1}^s \cos(\pi k_j y_j)
.
\end{eqnarray*}
Thus, by definition of $B_{\bsm}$, we have
\begin{eqnarray*}
  \lambda_s(B_{\bsm}) |\widetilde{f}_{\bsm,\cos}(\bszero)|^2 \beta^2
  & = &
  \int_{B_{\bsm}} \left( \frac{1}{2^s} \sum_{\uu \subseteq [s]}\sum_{n=0}^{N-1} w_n f_{\bsm}(\bsx_n \pmu \bsy) - \widetilde{f}_{\bsm,\cos}(\bszero) \beta \right)^2 \,\mathrm{d} \bsy \\
& \le & \int_{[0,1]^s} \left( \frac{1}{2^s} \sum_{\uu \subseteq [s]}\sum_{n=0}^{N-1} w_n f_{\bsm}(\bsx_n \pmu \bsy) - \widetilde{f}_{\bsm,\cos}(\bszero) \beta \right)^2 \,\mathrm{d} \bsy \\
& = & \sum_{\bsk \in \mathbb{N}_0^s \setminus \{\bszero\}} |\widetilde{f}_{\bsm,\cos}(\bsk)|^2 \,  \left( \sum_{n=0}^{N-1} w_n \prod_{j=1}^s \cos(\pi k_j x_{j,n}) \right)^2.
\end{eqnarray*}

We are now ready to piece this all together to obtain
\begin{eqnarray*}
\lefteqn{
 e^2(\mathcal{H}(K^{\cos}_{\alpha,\bsgamma,s}); P; \bsw)
 =
 (1-\beta)^2 + \sum_{\bsk \in \mathbb{N}_0^s \setminus \{\bszero\}} r_{\alpha,\bsgamma,s}(\bsk) \, 2^{|\bsk|_0} \left(\sum_{n=0}^{N-1} w_n \prod_{j=1}^s \cos(\pi k_j x_{j,n}) \right)^2
 }
 \\
 & \ge &
 (1-\beta)^2  + C_1(a,\bsgamma,s) \frac{2^{2t}}{N^{2\alpha}}  \sum_{\bsm \in \NN_0^s \atop |\bsm| = t} \sum_{\bsk \in \mathbb{N}_0^s \setminus \{\bszero\}} |\widetilde{f}_{\bsm,\cos}(\bsk)|^2 \, 2^{|\bsk|_0} \left(\sum_{n=0}^{N-1} w_n \prod_{j=1}^s \cos(\pi k_j x_{j,n}) \right)^2
 \\
 & \ge &
 (1-\beta)^2 +  C_1(a,\bsgamma,s) \frac{2^{2t}}{N^{2\alpha}} \sum_{\bsm \in \NN_0^s \atop |\bsm| = t} \lambda_s(B_{\bsm}) |\widetilde{f}_{\bsm,\cos}(\bszero)|^2 \beta^2
 \\
 & \ge &
 (1-\beta)^2 +  C_2(a,\bsgamma,s) \beta^2  \left(I(f)\right)^{2s} \frac{2^{2 t}}{N^{2\alpha}} 2^{-2t-4s} \sum_{\bsm \in \NN_0^s \atop |\bsm| = t} 1
 \\
 & \ge &
 (1-\beta)^2 +  C_3(\alpha,\bsgamma,s) \beta^2 N^{-2\alpha} {t+s-1 \choose s-1}.
\end{eqnarray*}
Set $A := C_3(a,\bsgamma,s) N^{-2\alpha} {t+s-1 \choose s-1}$. Then the last expression can be written as $(1-\beta)^2 + A \beta^2$, which satisfies
\begin{equation*}
e^2(\mathcal{H}(K^{\cos}_{\alpha,\bsgamma,s}); P; \bsw)
\ge
(1-\beta)^2 + A \beta^2 \ge \frac{\min(1,A)}{2} \ge C_4(\alpha,\bsgamma,s) N^{-2\alpha} {t+s-1 \choose s-1},
\end{equation*}
which implies the result, since $t \ge \log_2(N)$.
\end{proof}

\section{Numerical results}\label{sec_num_res}

In this section we show some numerical examples of applying lattice rules, tent-transformed lattice rules and symmetrized lattice rules to some test functions.
For this we use the lattice \emph{sequence} from \cite{HKKN2011} which was constructed to give $3$rd order convergence in a Korobov space.
It is a $10$-dimensional base~$2$ sequence with a maximum of $2^{20}$~points and is comprised of  embedded lattice rules with sizes $2^m$ for $m=0,\ldots,20$.
This lattice sequence was also used in \cite{NC2010} for some higher order convergence tests.

We report on two test functions to show some effects:
\begin{align*}
  g_{s,w}(\bsx)
  &:=
  \prod_{j=1}^s \left( 1 + \frac{w^j}{21} \left(-10 + 42 x_j^2 - 42 x_j^5 + 21 x_j^6\right) \right)
  ,
  \\
  h_{s,w}(\bsx)
  &:=
  \prod_{j=1}^s \left( 1 + \frac{w^j}{8} \left(31-84 x_j^2+8 x_j^3+70 x_j^4-28 x_j^6+8 x_j^7-16 \cos(1)-16 \sin(x_j)\right) \right)
  .
\end{align*}
Both functions integrate to $1$ over $[0,1]^s$.
The parameter $w$ acts like a product weight $w^j$.

All tests use $2^{20}$ (plus~$1$ for the symmetrized rule) function evaluations for their final result.
In Figure~\ref{fig:f1-f4} we report the actual running time in microseconds of optimized \Cpp\ code (to accommodate for the difference in which integration nodes are constructed).
Every mark represents an approximation with $2^m$ function evaluations or $2^{s-1}2^{m'} + 1$ for the symmetrized rule.
From this it can be seen that all three methods approximately have the same cost in this implementation, but in general this is dependent on the relative differences in the time for generating lattice sequence points, symmetrization and function evaluation.
For completeness we note that the tests were run on a 1.8~GHz Intel Core~i7 and compiled with the clang++~3.1 \Cpp\ compiler.
Tests on an older 2~GHz Intel Xeon compiled with the g++~4.8 compiler gave similar looking results (but slower).
The \Cpp\ source code and the raw data (for more test functions than shown here) can be downloaded from the KU~Leuven Lirias document repository.

The function $g_{s,w}(\bsx)$ is a product of $1 + w^j (B_6(x_j) + E_5(x_j))$, where $B_6$ is the degree~$6$ Bernoulli polynomial and $E_5$ is the degree~$5$ Euler polynomial.
From their Fourier expansions \cite{AS64} follows that we expect $3$rd order convergence for both the tent-transformed and symmetrized lattice sequence, while we only expect $1$st order convergence for the standard lattice sequence.
We show some results in the left hand column of Figure~\ref{fig:f1-f4} for $s=8$.
When the weight parameter is close to~$1$ ($w=0.9$ in the middle panel) we notice that the performance of the symmetrized and the tent-transformed rules is similar.
If $w$ is much smaller ($w=0.1$ in the top panel) then the tent-transformed rule wins, while if $w$ is larger ($w=2$ in the bottom panel) then the symmetrized rule wins; these effects are more pronounced when the dimension gets larger.

\begin{figure}
  \centering
  \includegraphics[width=0.47\textwidth]{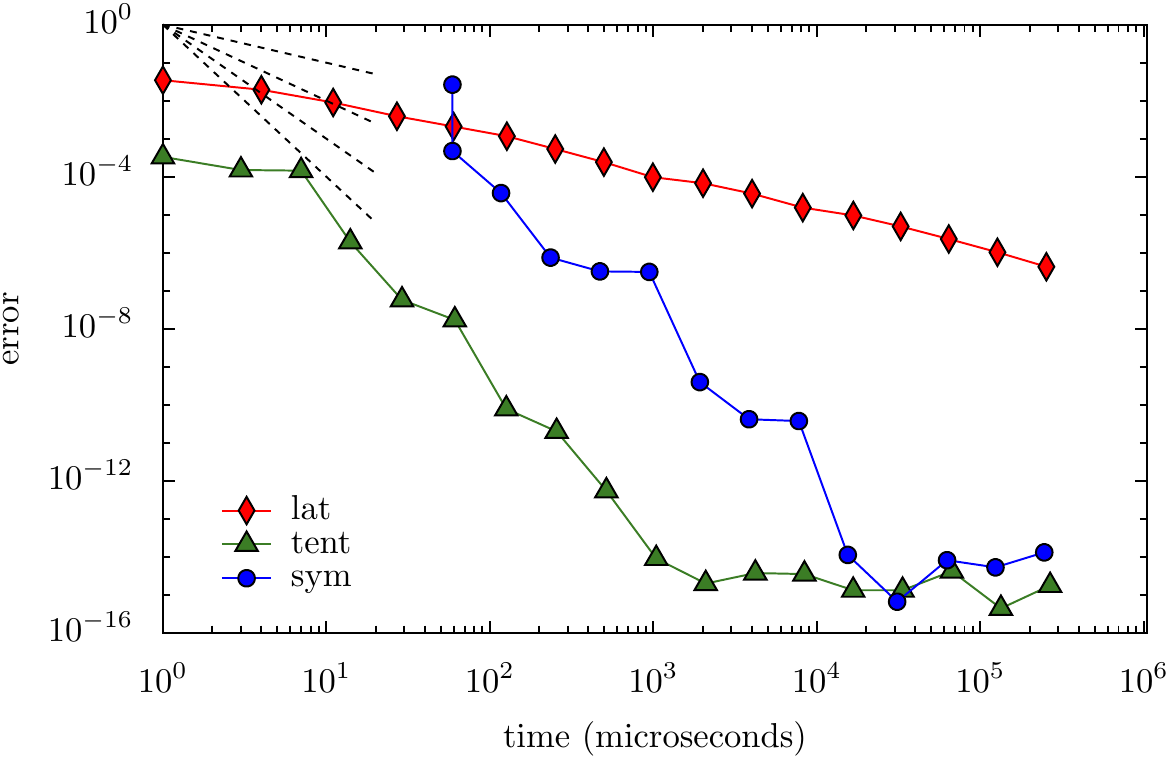}
  \includegraphics[width=0.47\textwidth]{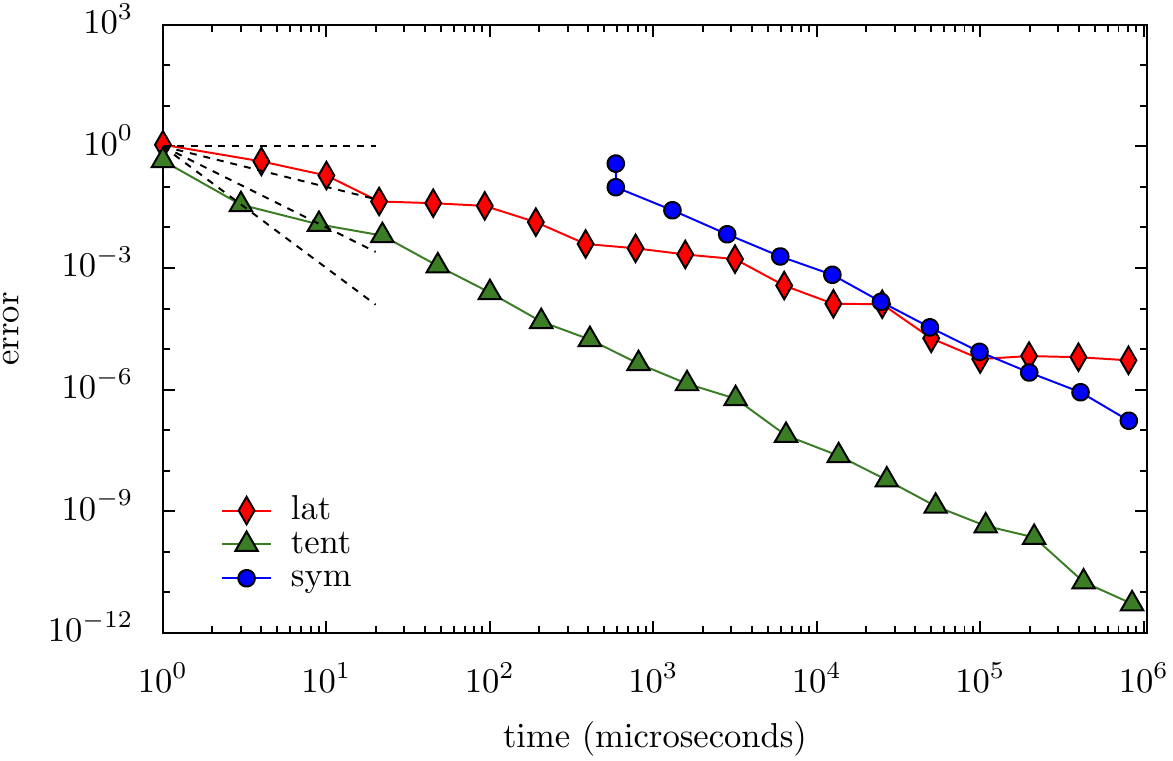}
  \\[2mm]
  \includegraphics[width=0.47\textwidth]{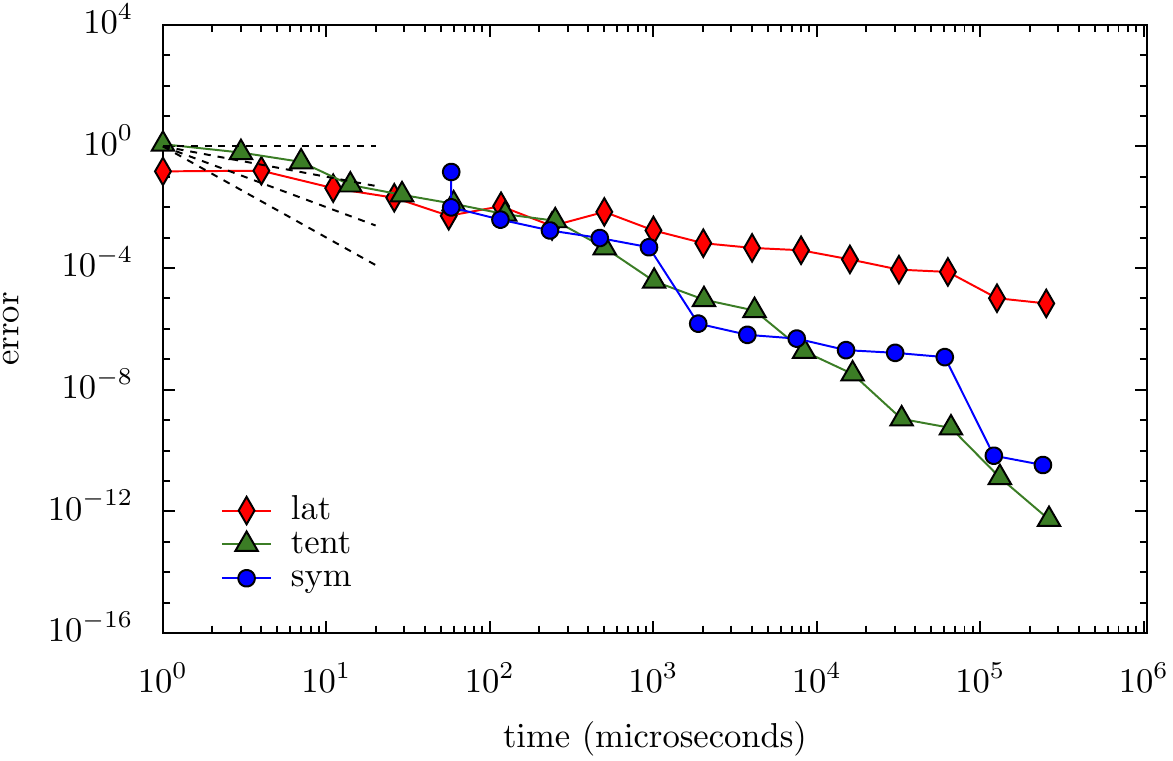}
  \includegraphics[width=0.47\textwidth]{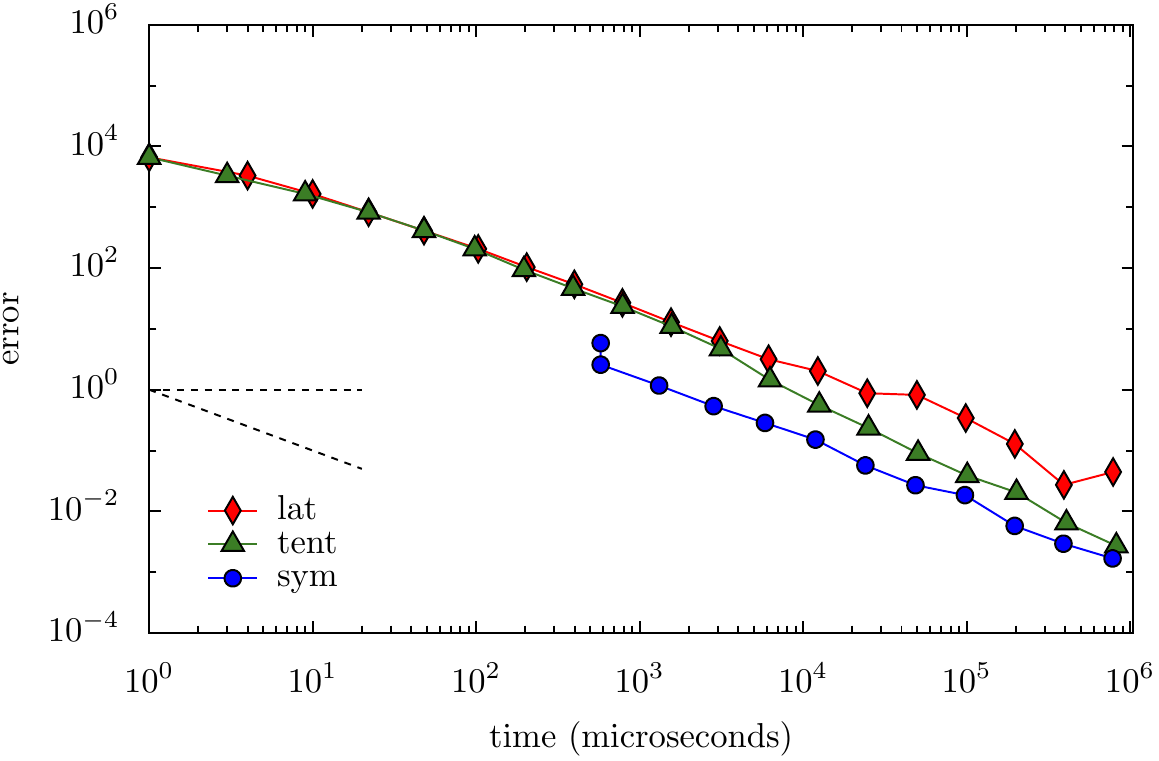}
  \\[2mm]
  \includegraphics[width=0.47\textwidth]{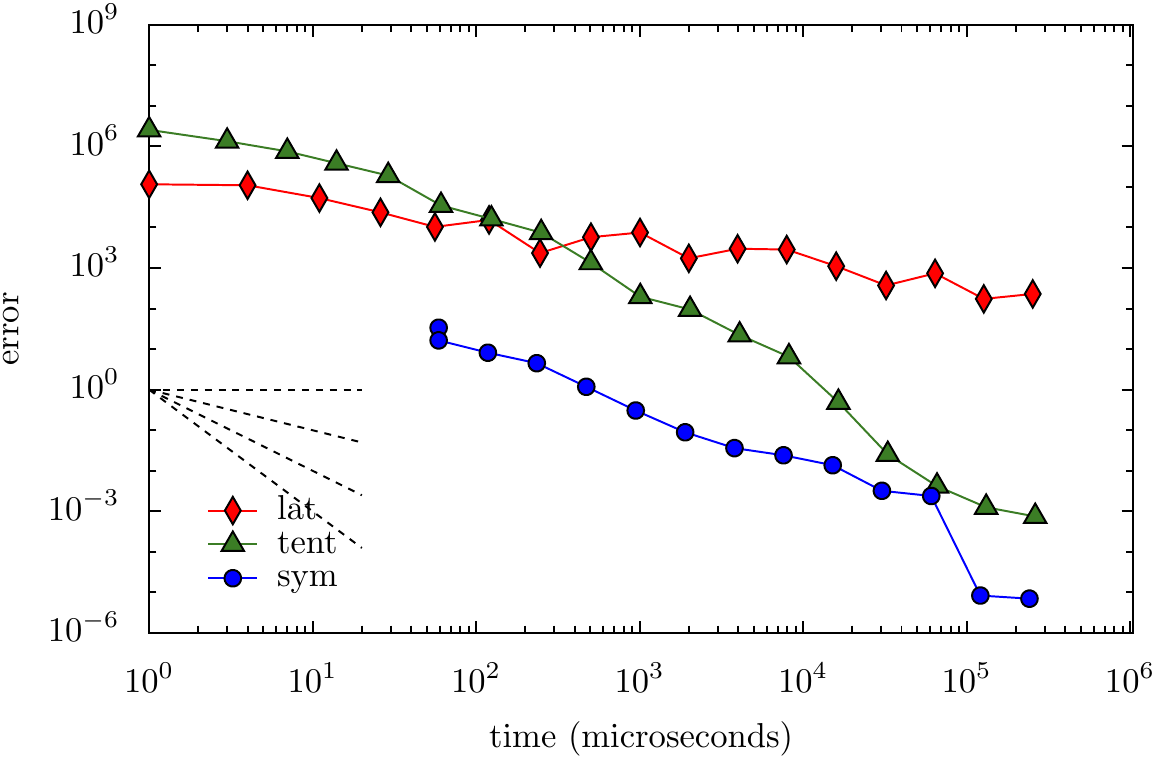}
  \includegraphics[width=0.47\textwidth]{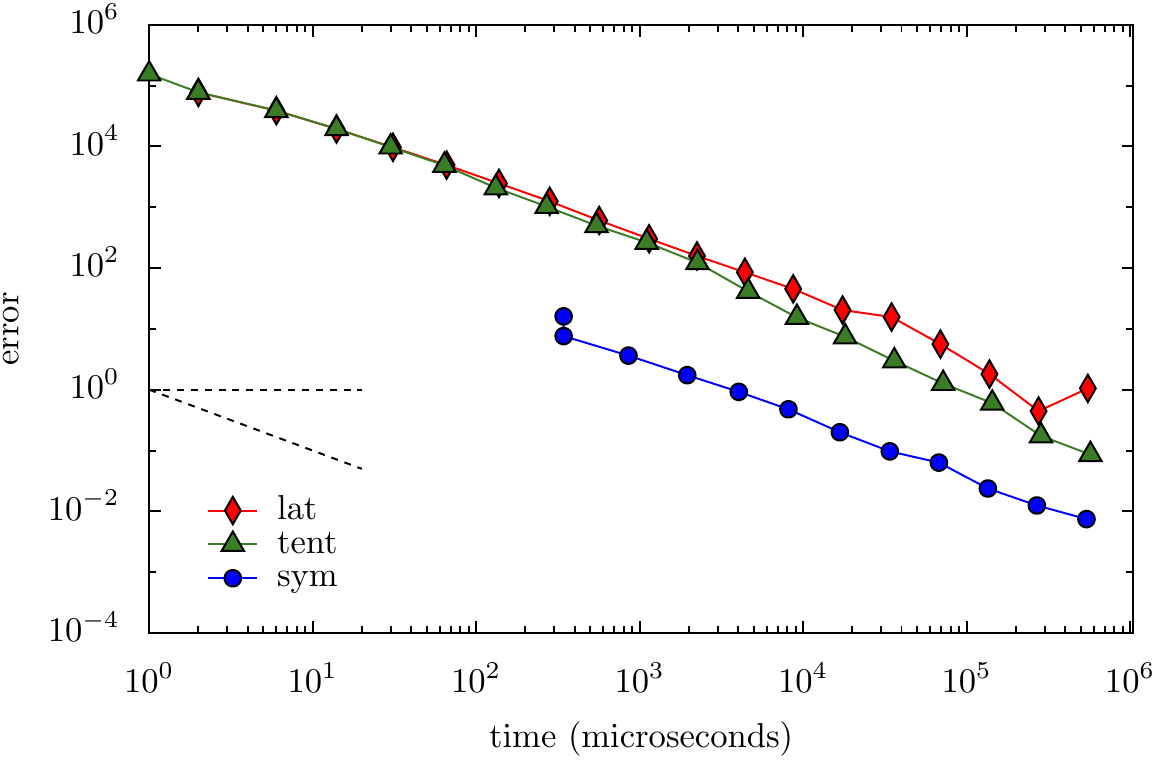}
  \caption{
    Left column: function $g_{s,w}(\bsx)$ in $8$~dimensions for $w=0.1$, $w=0.9$ and $w=2$.
    Right column: function $h_{s,w}(\bsx)$ in $10$~dimensions for $w=0.1$, $w=0.9$ and $w=1$.
  }\label{fig:f1-f4}
\end{figure}

The function $h_{s,w}(\bsx)$ is a more arbitrary function (including $\sin(x_j)$ and $E_7(x_j)$) for which its cosine expansion converges only like $k^{-2}$.
We thus expect $1$st order convergence for the tent-transformed and symmetrized methods.
For $s=10$ the behavior is illustrated in the right hand column of Figure~\ref{fig:f1-f4}.
Here again it can be seen that small values of $w$ give the advantage to the tent-transformed lattice rule, see the top panel where $w=0.1$.
However, already for $w=0.9$ the symmetrized rule outperforms both other methods.
The bottom panel shows the case $w=1$ where it can be seen that the number of points needed to reduce the error below~$1$ is already getting quite high.
In such a case there is an exponential dependence on the number of dimensions and the symmetrized rule is then the preferred choice if the dimension is small enough (this can also be seen in the left bottom panel which has $w=2$ for $g_{s,w}(\bsx)$ in $s=8$ dimensions). We remark that the graphs, where one replaces the time in the abscissa by the number of points $N$ in Figure~\ref{fig:f1-f4}, are very similar to the ones shown in Figure~\ref{fig:f1-f4}.

\section{Conclusion}\label{sec_conclusion}

In this paper we have shown that it is possible to obtain a
convergence rate of $N^{-\alpha + \delta}$ for any $\delta
> 0$ for sufficiently smooth integrands for lattice-type rules also for nonperiodic
functions. Previously, the only QMC rules with such
properties where higher order digital nets \cite{D08}. However, the function spaces we consider here are smaller than the usual smooth Sobolev spaces used for higher order digital nets.

Since for smoothness $1$ the unanchored Sobolev space and the half-period cosine space coincide, we obtain that tent-transformed lattice rules achieve the same convergence behavior and tractability results as lattice rules in Korobov spaces.
In contrast to previous results this technique does not need randomization.

\section{Acknowledgements}

J.D.\ is supported by an Australian Research Council Queen Elizabeth II fellowship.
D.N.\ is a fellow of the Research Foundation Flanders (FWO) and thanks Prof.\ Ian H.\ Sloan for initial discussions on the half-period cosine space.
The first two authors are grateful to the Hausdorff Institute in Bonn where most of this research was carried out.
F.P.\ is partially supported by the Austrian Science Foundation (FWF), Project S9609.


\bigskip

\begin{small}
\noindent\textbf{Authors' addresses:}
\\ \\
\noindent Josef Dick,
\\
School of Mathematics and Statistics,
University of New South Wales, Sydney, NSW, 2052, Australia\\
\\
\noindent Dirk Nuyens,
\\
Department of Computer Science,
KU Leuven, Celestijnenlaan 200A, 3001 Heverlee, Belgium\\
\\
\noindent Friedrich Pillichshammer,
\\
Institut f\"{u}r Finanzmathematik,
Universit\"{a}t Linz, Altenbergerstr.~69, 4040 Linz, Austria\\
\\

\noindent \textbf{E-mail:} \\
\texttt{josef.dick@unsw.edu.au}\\
\texttt{dirk.nuyens@cs.kuleuven.be}\\
\texttt{friedrich.pillichshammer@jku.at}
\end{small}

\end{document}